\documentclass{amsart}

\usepackage{amsmath}
\usepackage{amsfonts}
\usepackage{amssymb}
\usepackage{amsthm}
\usepackage{amscd}
\usepackage[all,cmtip, line]{xy}

\newtheorem{thm}{Theorem}
\newtheorem*{thmA*}{Theorem A}
\newtheorem*{thmAB*}{Theorem 30 (Theorem A)}
\newtheorem{lem}[thm]{Lemma}
\newtheorem{prop}[thm]{Proposition}

\newtheorem*{defin}{Definition}

\newcommand{\su}{Gl_2(\Q)\backslash(\mathbb{C}\backslash\mathbb{R})\times\glaf
/U}

\newcommand{\gla}{GL_2(\mathbf{A})}
\newcommand{\glaf}{GL_2(\af)}
\newcommand{\glafp}{GL_2(\af^p)}
\newcommand{\upi}{U^p \times \Gamma_1(p^n)}
\newcommand{\upq}{U^p \times \Gamma_1(\mathbf{q})}
\newcommand{\gl}{G_{\mathbb{Q}_l}}
\newcommand{\glzp}{GL_2(\mathbb{Z}_p)}
\newcommand{\glzl}{GL_2(\mathbb{Z}_l)}
\newcommand{\glzh}{GL_2(\hat{\mathbb{Z}})}
\newcommand{\glpl}{GL_2^+(\mathbb{Q})}
\newcommand{\glql}{GL_2(\Q_l)}
\newcommand{\wl}{W_{\mathbb{Q}_l}}
\newcommand{\il}{I_{\mathbb{Q}_l}}
\newcommand{\zh}{\hat{\mathbb{Z}}^*}

\newcommand{\Zp}{\mathbb{Z}_p}
\newcommand{\Qp}{\mathbb{Q}_p}
\newcommand{\af}{\mathbf{A}_f}
\newcommand{\Q}{\mathbb{Q}}

\newcommand{\Cp}{\mathbb{C}_p}
\newcommand{\B}{\mathcal{B}}
\newcommand{\Bs}{\mathcal{B}^*}

\newcommand{\xu}{X(U(N)^p\times \Gamma_1(\mathbf{q}))_{\geq p^{-r}}}
\newcommand{\xz}{X(U(N)^p\times \Gamma_0(\mathbf{q}))_{\geq p^{-r}}}

\begin{document}
\title{\textbf{LOCAL TO GLOBAL COMPATIBILITY ON THE EIGENCURVE}
$(l\neq p)$}
\author{Alexander Paulin}
\begin{abstract}
We generalise Coleman's construction of Hecke operators to define an action
of $\glql$ on the space of finite slope overconvergent $p$-adic modular forms
($l\neq p$).  In
this way we associate to any $\Cp$-valued point on the tame level $N$ Coleman-Mazur eigencurve
an admissible smooth representation of
$\glql$ extending the classical construction.  Using the Galois theoretic
interpretation of the eigencurve we associate a 2-dimensional Weil-Deligne
representation to such points and show that away from a discrete set they
agree under the Local Langlands correspondence.
\end{abstract}
\maketitle
\tableofcontents
\section{Introduction}
In the mid 1990s, Coleman wrote a series of papers (see \cite{CCO}, \cite{CME},
\cite{BMF}) which rejuvenated the study of $p$-adic properties of modular
forms as initiated by Katz (\cite{K}) and Serre (\cite{S}) almost thirty
years earlier.  Perhaps the most profound output of this work was the construction
of the eigencurve (\cite{CME}); a type of universal parameter space for both
overconvergent $p$-adic eigenforms and residually modular Galois representations.
In recent
years much work has been done to understand how these objects behave
over the eigencurve.    For example, work of Kisin {\cite{KIS}) applies Fontaine
theory to $p$-adic Galois representations on the eigencurve to deduce certain cases of the
Fontaine-Mazur conjecture.  In this paper we study how automorphic representations
attached to classical cusp
forms (away from $p$ and $\infty$) vary across the eigencurve.

More precisely, fix a prime $p$ and let $\mathbf{A}$ denote the adeles over $\mathbb{Q}$. To $f \in \mathcal{S}_k(\Gamma_1(N),\chi)$, a normalised classical cuspidal eigenform, we can associate (non-canonically)
a cuspidal algebraic automorphic representation, $\pi_f$, of $GL_2(\mathbf{A})$. The finite part of this representation in naturally the restricted tensor
product of smooth irreducible representations $\pi_{f,l}$ of $\glql$
 as $l$ ranges over all rational primes. 

Also associated to $f$ is a compatible system of Galois representations: $$\rho_f: G_{\mathbb{Q}} \rightarrow GL_2(\bar{\mathbb{Q}}_p).$$  The global Langlands philosophy
links these two objects, but the correspondence can be seen at the local
level.  If we restrict $\rho_f$ to a decomposition group at $l\neq p$ then
Grothendieck's abstract monodromy theorem attachs to $f$ a Weil-Deligne representation
$(\rho_{(f,l)}, N_f)$.
Work of Carayol (\cite{C}) shows that this matches with $\pi_{f,l}$ under the (correctly normalised) local Langlands correspondence.
This  expresses the compatibility of the local and global Langlands correspondences.
In this paper we generalise this result to overconvergent $p$-adic eigenforms.

Let $\mathcal{E}$ denote the reduced tame level $N$ cuspidal eigencurve (as constructed
in \cite{CME}, \cite{BE}). To any $\Cp$-valued point (equivalent to a overconvergent
eigenform of tame level $N$ of some weight $\kappa$) we associate a admissible smooth (not necessarily irreducible) representation
of $\glql$ for   $l\neq p$.  This is done by defining a \textit{weight} zero
action using the moduli problem of elliptic curves and then introducing a weight $\kappa$
twist.  As in Coleman's construction of Hecke operators this twist factor
is obtained by $p$-adically deforming the $q$-expansions of classical Eisenstein series. In the study of the $p$-adic properties of automorphic forms this is
a common theme - embedding objects in large ambient $p$-adic Banach space
which come equipped with some kind of \text{weight} zero action and then
defining arbitrary weights by introducing certain twist factors (see \cite{CHE},
\cite{BE}).

The Galois theoretic construction of $\mathcal{E}$ ensures that there is
an admissible open dense subspace $\mathcal{E}^o \subset \mathcal{E}$ together with
an admissible
covering by open affinoids $\mathcal{U}$ such that for any $Sp(A) \in \mathcal{U}$
there exists a continuous Galois representation
$$\rho: G_{\mathbb{Q}} \rightarrow GL_2(A).$$
This representation is unique up to isomorphism over $A$.
$\mathcal{E} \backslash \mathcal{E}^o$ is the discrete set  corresponding
to pseudocharacters which give rise to reducible representations.   $\mathcal{E}^o$ contains all the classical cuspidal
locus, which forms a Zariski dense set.  If $\mathcal{Z} \subset \mathcal{E}^o$ 
is an irreducible
component (see \cite{BC}) then there is naturally an admissible cover of
$\mathcal{Z}$ with the above property. 
A mild generalisation of Grothendieck's monodromy theorem allows us to construct
a family of two dimensional Weil-Deligne representations across $\mathcal{Z}$,
which agrees with the ordinary construction on the classical locus. Thus
for any $\kappa \in   \mathcal{E}^o (\Cp)$ we have $\pi_{f_{\kappa},l}$, an  admissible
smooth
representation of $\glql$ and a two dimensional  Weil-Deligne representation
$(\rho_{(\kappa,l)}, N_{\kappa})$. Let $\pi_m$ be the modification of the
Tate normalised local Langlands correspondence introduced in  \S5.3.
\\
Our main result is the following:
\begin{thmA*}
Away from a discrete set of points local to global compatibility holds on
$\mathcal{E}^o$, i.e.
$$\pi_m(\rho_{\kappa}, N_{\kappa}) \cong \pi_{f_{\kappa},l}$$ for all $\kappa
\in \mathcal{E}^o(\Cp)$ away from some discrete set.  More precisely if $\mathcal{Z}
\subset \mathcal{E}^o$ (taken over $\Cp$) is an irreducible component then either:
\begin{enumerate}
\item $\mathcal{Z}$ is Supercuspidal.
In this case local to global compatibility holds on all $\mathcal{Z}$.
\item $\mathcal{Z}$ is Special. In this case local
to global compatibility holds everywhere except at points where monodromy
vanishes.  For such $\kappa$, $\pi_{f_{\kappa},l}$ is the unique special irreducible sub-representation of $\pi_m(\rho_{\kappa}, N_{\kappa})$.
\item $\mathcal{Z}$ is Principal Series.  Here local to global compatibility holds except at points where
the ratio of the Satake parameters becomes $l^{ \pm 1}$.  At such points all
we know is that these is a smooth admissible representation $\pi$ and a $\glql$
equivariant
surjection $\pi \rightarrow   \pi_{f_{\kappa},l}$ where the semi-simplification of
$\pi$ if isomorphic to the semi-simplification of $\pi_m(\rho_{\kappa}, N_{\kappa})$.
\end{enumerate}
\end{thmA*}

We briefly outline the proof:  Fix a sufficiently small open compact subgroup
$U \vartriangleleft \glzl$ which is \textit{bonne} in the sense of \cite{CB} and such
that $\pi_{f_{\kappa},l}^U$ is non-zero for all $\kappa \in \mathcal{E}^o(\Cp)$. Also
fix a Haar measure and let $\mathcal{H}_U$ denote the Hecke algebra of compactly
supported $\Cp$-valued
$U$ bi-invariant functions on $\glql$. One of the central results of \cite{CB}
is that there
is an equivalence of categories between admissible smooth representations
of $\glql$ and finite dimensional $\mathcal{H}_U$ modules. Because we are
working over $\Cp$ such a module is determined up to semisimplicity by traces.
Let $\widetilde{\mathcal{E}^o}$ be a normalisation of $\mathcal{E}^o$, considered
as rigid spaces over $\Cp$.  Using the automorphic interpretation of $\mathcal{E}$ we construct a
function
$$\xymatrix{ tr_{aut}:\mathcal{H}_U
\ar[r]
&O(\widetilde{\mathcal{E}^o})},$$
which agrees with the trace map of $\mathcal{H}_U$ acting on $\pi_{f_{\kappa},l}^U$
away from a discrete subset.  Similarly, using the Galois interpretation of $\mathcal{E}$ we construct a
function 
$$\xymatrix{ tr_{Lan}:\mathcal{H}_U
\ar[r]
&O(\widetilde{\mathcal{E}^o})},$$
which agrees with the trace map of $\mathcal{H}_U$ acting on $\pi_m(\rho_{\kappa}, N_{\kappa})^U$ away from a discrete set.  
By Coleman's classicality result and classical local to global compatibility
we deduce that both functions must agree on a Zariski dense set.  Hence we
deduce that $tr_{aut} = tr_{Lan}$.  This proves local to global compatibility
on $\mathcal{E}$ away from a discrete set.  The rest of Theorem A is deduced
from a careful study of the finer properties of both these functions.
\\ \\
This paper in the authors PhD thesis under the supervision of Kevin Buzzard,
and it is a pleasure to thank him both for suggesting the problem and for
the numerous helpful conversations.  I would also like to thank Gaetan Cheveier
and 
Matthew Emerton for several very helpful conversations.

\section{The Classical Case}
We  briefly recall how to attach an automorphic representation to a cuspidal eigenform.
\\ \\
Let us fix the following notation:
\\ \\
$\mathbf{A}$: the adeles over $\Q$
\\ \\
$\af$: the finite adeles over $\Q$
\\ \\
$\hat{\mathbb{Z}}:=  \prod_{p}\mathbb{Z}_p \subset \af$
\\ \\ 
$U \subset \glaf$: an open compact subgroup
\\ \\
$U_0(N):= \{g \in Gl_2(\hat{\mathbb{Z}}) | g \equiv \begin{pmatrix}* &* \\
0 &*\end{pmatrix} (N) \}$
\\ \\
$U_1(N):= \{g \in Gl_2(\hat{\mathbb{Z}}) | g \equiv \begin{pmatrix}* &* \\
0 &1\end{pmatrix} (N) \}$
\\ \\
$U(N):= \{g \in Gl_2(\hat{\mathbb{Z}}) | g \equiv \begin{pmatrix}1 &0 \\
0 &1\end{pmatrix} (N) \}$
\\ \\
We can decompose $Gl_2(\mathbf{A})$ with respect to $U$ using the following
famous result of Borel:
\begin{lem}
If $U \subset Gl_2(\af)$ is an open, compact subgroup then $$Gl_2(\mathbf{A})=\coprod_{j=1}^r
Gl_2(\Q)\cdot g_j U\cdot Gl_2^+(\mathbb{R})\;\text{where}\;g_j \in Gl_2(\mathbf{A}),
\; r< \infty$$ If $det(U)=\hat{\mathbb{Z}^*}$
then $r=1$ and $g_1$ can be trivial.\begin{proof}
This is just an application of the strong approximation theorem.
\end{proof}
\end{lem}
Let $\Gamma \in GL_2^+(\mathbb{Q})$ be an arithmetic subgroup such that $U
\cap \glpl = \Gamma$.  Let us further assume that $det(U)=\zh$. We make the
following important defintion:
\begin{defin}
For $f\in \mathcal{M}_k(\Gamma)$ a classical modular form of weight $k \in \mathbb{Z}$ and level $\Gamma$
define
$$\varphi_f:GL_2(\mathbb{Q})\backslash GL_2(\mathbf{A})/U\rightarrow \mathbb{C}$$ $$[g = \gamma u \delta] \longrightarrow
f(\delta i)j(\delta , i )^{-k}det(\delta)$$ Where $\gamma \in GL_2(\Q)$,
$u \in U$ and $\delta \in GL_2^+(\mathbb{R})$, as in the decomposition
of lemma $1$
\end{defin}
This is not a canonical procedure.  There are different conventions about
what power of determinant to choose. This choice ensures that classical and
adelic Hecke operators (when correctly normalised) agree. We embed such functions in the ambient space
of weight $k$ modular forms:

$$\mathcal{M}_k:=\{\varphi:Gl_2(\Q)\backslash Gl_2(\mathbf{A}) \rightarrow
\mathbb{C} \; | \; (*) \}$$
Where $(*)$ is a set of conditions:
\begin{enumerate}

\item $\varphi$ must
be left invariant under some open compact subgroup $U \subset Gl_2(\af)$.
\item $\forall g \in Gl_2(\af)$ we have a map $f_g^{\varphi}:\mathcal{H} \rightarrow
\mathcal{C}$ given by:   $$f_g^{\varphi}(\delta i) =
\varphi(g\delta)j(\delta,i)^{k}det(\delta)^{-1}$$ where $\delta \in Gl_2^+(\mathbb{R})$.
This must be well defined and holomorphic.
\item $\varphi $ is slowly increasing, see \cite{DT}.
\end{enumerate}
$\mathcal{M}_k \subset \mathcal{A}$, where $\mathcal{A}$ is the space of automporphic
forms for $GL_{2/{\mathbb{Q}}}$.  It naturally comes equipped with a smooth
admissible action $\glaf$. There is a very direct relation between this space
and classical modular forms. If $\varphi \in \mathcal{M}_k^U$
then $f_g^{\varphi} \in \mathcal{M}_k(\Gamma_g)$
where $\Gamma_g = gUg^{-1}\cap \glpl$.  Condition $(ii)$ assures us that
it satisfies the automorphic invariance property with respect to $\Gamma_g$,
and condition $(iii)$ assures us that it is well behaved at cusps. If 
$$\glaf = \coprod_{j=1}^{r}GL_2(\Q)\cdot g_j \cdot U 
\text{ where } g_j \in \glaf, \; r< \infty$$
and $\Gamma_j =
g_j U g_j^{-1} \cap GL_2^+(\Q)$  then there is an isomorphism of complex
vector spaces $$\mathcal{M}_k^U \rightarrow
\bigoplus_{j=1}^{r}\mathcal{M}_k(\Gamma_j)$$
$$\varphi \rightarrow (f_{g_j}^{\varphi}).$$
In this way $\mathcal{M}_k^{U_1(N)} \cong \mathcal{M}_k(\Gamma_1(N))$.  If
we include a nebentypus character $\chi$ then there is an isomorphism.
$\mathcal{M}_k^{U_1(N)}(\chi) \cong \mathcal{M}_k(\Gamma_1(N), \chi^{-1})$
Our choice of
conventions ensures that with is an isomorphism of Hecke modules away from
$N$.

\begin{defin}
For $\varphi \in \mathcal{M}_k$ and $ g \in \glaf$ we define the $g$-$q$-expansion
of $\varphi$ as being the $q$-expansion at $\infty$ of $f_g^{\varphi}$.
\end{defin}
From the above it is clear that any  $\varphi \in \mathcal{M}_k$ is uniquely determined by a finite number of these $q$-expansions.

The space of cusp forms can also be expressed in this adelic language and
we obtain the space $\mathcal{S}_k \subset \mathcal{M}_k$ by demanding a
\textit{cuspidal} condition, see \cite{DT}.  This property means that  $f_g^{\varphi} \in S_k(\Gamma_g)$.   $\mathcal{S}_k \subset \mathcal{A}^o$, the space
of cuspidal automorphic forms.  This latter space decomposes discretely into
irreducible automorphic representations of $\gla$ each with multiplicity
one.  In the usual way, this is really a $(\mathfrak{gl_2}, O_2(\mathbb{R}))$-
module at infinity
but we suppress this from the notation.  For $k \geq 2$ the span of $\mathcal{S}_k$
within $\mathcal{A}^o$ is precisely the direct sum of the automorphic representations
whose infinite component is $\mathcal{D}_k$, the  holomorphic weight $k$ discrete series (with some fixed central character that we will suppress
from the notation).  In this way, $\mathcal{S}_k$
decomposes as the  direct sum of the irreducible admissible representations
of $\glaf$ appearing in $\mathcal{A}^o$ whose infinity component is $\mathcal{D}_k$ .  If  $f \in S_k(\Gamma_1(N),
\chi )$ is a newform then $\varphi_f$ lies in a unique such irreducible constituent.
This sets us a bijection between newforms and irreducible subrepresentations
of $\mathcal{S}_k$.  More generally, any eigenform $f$ lies in a unique irreducible
sub-representation.
\begin{defin}
If $f \in S_k(\Gamma_1(N),
\chi )$ is an eigenform we define $$\pi_f = \mathbb{C}[\glaf]\varphi_f \subset
\mathcal{S}_k$$
as the irreducible smooth automorphic representation attached to $f$.
\end{defin}
By the general theory of such representations we know that it decomposes
as the restricted tensor product of local representations:
$$\pi_f =\grave{\bigotimes_l} \pi_{f,l},$$
where the product is over all rational primes and the local factors are smooth irreducible representations of $\glql$ in the sense of \cite{JL}.  Clearly
$\pi_{f,l} \cong \mathbb{C}[\glql]\varphi_f$.  
\\ \\
The theory overconvergent
modular forms is built on the $p$-adic properties of $q$-expansions.  If
$f$ is an overconvergent eigenform in the sense of \cite{CME} then to attach
a representation of $\glql$ we must make sense of this action at the level
of $q$-expansions. 
\begin{prop}
Let $U\in \glaf$
an open compact subgroups, such that  $det(U)=\zh$.
Let $g, h \in \glaf$ with $hg=\gamma \cdot u$ where $\gamma \in Gl_2^+(\Q)$ and
$u\in U$.  For any $\varphi \in \mathcal{M}_k^U$  the $h$-$q$-expansion at of $g(\varphi)$ is the
$q$-expansion at infinity of $det\gamma^{k-2}\cdot f_{1}^{\varphi}|_k\gamma^{-1}$.

\begin{proof} 
The conventions we choose are that  $f|_k\gamma(\tau):=(det\gamma)^{k-1}\cdot
j(\gamma, \tau)^{-k}\cdot f(\gamma \tau).$  By definition $f^{g(\varphi)}_h (\delta i) =
\varphi(h\delta g)j(\delta,i)^{k}det(\delta)^{-1} =
\varphi(\gamma\delta)j(\delta,i)^{k}det(\delta)^{-1}.$ $\varphi$ is left
invariant by $GL_2(\mathbb{Q})$, so this simplifies to $\varphi(\gamma^{-1}\delta)j(\delta,i)^{k}det(\delta)^{-1}$.
By the elementary properties of the $j$ function this simplifies to the expression
we are looking for.
\end{proof}

\end{prop}
Ultimately we will be interested on the effect of this action on the $q$-expansions
of Eisenstein series as it is these which $p$-adically interpolate. We want
therefore to be able to explicitly determine the effect of this action on $q$-expansion of classical forms of level $U_1(p)$ and character $\chi$
under the action of the Diamond operators at $p$.  From now on fix a rational
prime $l \neq p$.  

\begin{lem}
Let $g \in \glql \in \glaf$.  Then $$ g = \begin{pmatrix}l^m
&\frac{a}{l^r} \\ 0 &l^n \end{pmatrix}\cdot u , \text{ with } a, r, m, n \in
\mathbb{Z}, u \in U_0(p)$$
Where the first term is embedded diagonally in $\glaf$.
\end{lem}
\begin{proof}
By the Iwasawa decomposition  $$\glql
= B(\Q)\glzl $$
where $B(\Q)$ is the Borel subgroup of upper triangular matrices.  Therefore
$$g=\begin{pmatrix}A
&B \\ 0 &C \end{pmatrix}\cdot v \text{ where }A, B, C \in \Q \text{ and } v \in \glzl. $$
Post-multiplying the first term by an an appropriate
diagonal matrix with coefficients in $\Q$ contained in $\glzl$ we can get
it into the form $$\begin{pmatrix}l^m
&\frac{a}{l^rN} \\ 0 &l^n \end{pmatrix} v \text{ where } m, n, r, N ,a \in \mathbb{Z},
r \geq 0,(N,l)=1
$$
We want $N=1$. Assume this is not the case and consider the following:
$$\begin{pmatrix}l^m
&\frac{a}{l^rN} \\ 0 &l^n \end{pmatrix}\cdot \begin{pmatrix}1
&\frac{c}{N} \\ 0 &1 \end{pmatrix}=\begin{pmatrix}l^m
&\frac{a}{l^rN}+\frac{l^mc}{N} \\ 0 &l^n \end{pmatrix}, c\in \mathbb{Z}$$
Note that the second factor is in $\glzl$ for any such choice of $c$, and
we may assume $r+m \geq0$. We are reduced to finding $c \in \mathbb{Z}$ such that $N | a + c l^{r+m}$. But $l$ and $N$ are coprime so  Euclid's algorithm assures us that there is such
a $c$.  Hence we are done.
\end{proof}

\begin{prop}
Let  $\varphi \in \mathcal{M}_k^{U_1(p)} (\chi)$, $g \in \glql$, $h \in U_0(p)$,
$h_l$ its component at $l$, and $h^l$ the trivialisation of $h$ at $l$. Let $$ h_l g = \begin{pmatrix}l^m
&\frac{a}{l^r} \\ 0 &l^n \end{pmatrix}\cdot u , \text{ with } a, r, m, n \in
\mathbb{Z}, u \in U_0(p)$$
If
 the $1$-$q$-expansion at infinity
of $\varphi$ is $f_1^{\varphi}(q)$ then the $h$-$q$-expansion at infinity of $g(\varphi)$ 
is 
$$f^{g(\varphi)}_h(q)=\chi( u h^l) l^{-(m+n)+nk}f_1^{\varphi}(q^{l^{n-m}}\cdot
e^{2\pi i (\frac{-a}{l^{m+r}})})$$

\begin{proof}
Let $\begin{pmatrix}l^m
&\frac{a}{l^r} \\ 0 &l^n \end{pmatrix} = \gamma$. By definition $$f^{g(\varphi)}_h (\delta i) =
\varphi(h\delta g)j(\delta,i)^{k}det(\delta)^{-1} = \varphi(h_lg\delta h^l)j(\delta,i)^{k}det(\delta)^{-1}$$
This final term is equal to $$\varphi(\gamma^{-1}\delta u h^l)j(\delta,i)^{k}det(\delta)^{-1}=\chi(u
h^l)\varphi(\gamma^{-1}\delta )j(\delta,i)^{k}det(\delta)^{-1}.$$
The proof of proposition 2 tells us that this term is equal to
$$\chi(u
h^l)det\gamma^{k-2}\cdot f_{1}^{\varphi}|_k\gamma^{-1}$$
$$f_1^{\varphi}|_k\begin{pmatrix}l^m
&\frac{a}{l^r} \\ 0 &l^n \end{pmatrix}^{-1} =f_1^{\varphi}|_k\begin{pmatrix}l^{-m}
&\frac{-a}{l^{r+m+n}} \\ 0 &l^{-n} \end{pmatrix}. $$  If $\tau$ is the parameter
on the upper half plane then we know that
$$ f_1^{\varphi}|_k\begin{pmatrix}l^{-m}
&\frac{-a}{l^{r+m+n}} \\ 0 &l^{-n} \end{pmatrix}(\tau)=l^{-(k-1)(m+n)}\cdot
l^{nk}\cdot f_1^{\varphi}(\frac{l^{-m}\tau +\frac{-a}{l^{r+m+n}}}{l^{-n}}) $$ Recalling
that $q=e^{2\pi i \tau}$ we know that $$ f_1^{\varphi}|_k\begin{pmatrix}l^{-m}
&\frac{-a}{l^{r+m+n}} \\ 0 &l^{-n} \end{pmatrix}(q)=l^{-(k-1)(m+n)+nk}f_1^{\varphi}(q^{l^{n-m}}\cdot
e^{2\pi i\frac{-a}{l^{m+r}}})$$  
Multiplying through by the constant term in proposition 2 yields the result.
\end{proof}
\end{prop}

We wish to reinterpret these results geometrically. 
Let us fix the following notation:
\\ \\
$\mathcal{U}:=\{U\subset \glaf | \text{ open, compact subgroup  s.t. } U\cap\glpl
\text{ is torsion free }\}$
\\ \\
$\Sigma:=\{U \in\mathcal{U} | U \subset \glzh\}$
\\ \\
For
$U\in\mathcal{U}$ we define the complex Shimura variety:
$$ Y_U=\su$$
Here $Gl_2(\Q)$ acts on $\mathbb{C}\backslash\mathbb{R}$ by Mobius transformations and
$\glaf$ by left translation.  $U$ acts trivially on $\mathbb{C}\backslash\mathbb{R}$
and by right translation on $\glaf$. If 
$$\glaf = \coprod_{j=1}^{r}Gl_2(\Q)\cdot g_j U 
\text{ where } g_j \in \glaf$$
and
$\Gamma_j =
g_j U g_j^{-1} \cap Gl_2^+(\Q)$.
Then there is a homeomorphism of topological spaces 
$$\coprod_{j=1}^{r}\Gamma_j\backslash\mathcal{H}\longrightarrow Y_U$$
$$\Gamma_j\backslash\mathcal{H}\longrightarrow Y_U$$
$$[x] \longrightarrow [x, g_j]$$
Hence $Y_U$ inherits the structure of a riemann surface. Ordering $\mathcal{U}$ by inclusion gives the projective system
$$\varprojlim_{\mathcal{U}}Y_U$$
$\glaf$ naturally acts on this projective system.  For $U \in \Sigma$, $Y_U$
is naturally the complex points of a non-singular algebraic  curve $Y(U)$ over $\mathbb{C}$,
with a (canonical, after fixing a Shimura datum) model over $\mathbb{Q}$, which
is the moduli space of elliptic curve with level $U$ structure, see \cite{DR}.
The push forward of the differential sheaf on the universal elliptic curve gives
an invertible sheaf $\omega_U$. $Y(U)$ has a canonical compactification $X(U)$
and this sheaf uniquely extends to a sheaf  on it which we also denote $\omega_U$.
This naturally
gives rise to the direct limit 
$$\varinjlim_{\Sigma}\Gamma(X(U), \;\omega_U^{\otimes k})$$ for $k \in
\mathbb{Z}$, which inherits an action of $\glaf$.  There is a natural identification $$ \mathcal{M}^U_k \cong \Gamma(X(U), \omega_U^{\otimes k}).$$ This gives
rise to an isomorphism of complex vector spaces
$$ \mathcal{M}_k \cong \varinjlim_{\Sigma}\Gamma(X(U), \;\omega_U^{\otimes k}) \subset \varinjlim_{\Sigma}\Gamma(Y(U), \;\omega_U^{\otimes k})$$
\begin{prop}
This is  an isomorphism of $\glaf$ modules 
\end{prop}
\begin{proof}This is very well known, but for want for a reference we prove
it.  It is sufficient to prove that the inclusion of $\mathcal{M}_k$ in $\varinjlim_{\Sigma}\Gamma(Y(U), \;\omega_U^{\otimes k})$ respects the $\glaf$ action.  We can work entirely
in the complex analytic category and consider global section of $\omega_U^{\otimes k}$ as being global sections of a line bundle $\mathcal{L}_U^k$.  We can take the direct limit as being over all of $\mathcal{U}$.

We are able to decompose $\glaf$ with respect to $U$,
getting double coset representatives $\{g_j\}$.  Using these, we are deduce
that $\{g_jg^{-1}\}$ are a set of double coset representatives with respect
to $gUg^{-1}$. Notice that
$$\Gamma_j' = \glpl \cap g_jg^{-1}(gUg^{-1})gg_j^{-1} = \Gamma_j$$
We have the isomorphism of riemann surfaces
$$
\xymatrix{
\coprod_{j=1}^r\Gamma_j\backslash\mathcal{H} \ar[r] &Y_{gUg^{-1}}\\
[x_j]\ar[r]&[x_j, \;g_jg^{-1}]
}$$
On connected components we get the maps
$$
\xymatrix{
(Y_{gUg^{-1}})_j \ar[r]^g &(Y_U)_j \ar[d] & &[x_j,\;g_jg^{-1}] \ar[r]^g &[x_j,\;g_j]\ar[d]\\
\Gamma_j\backslash\mathcal{H} \ar[u] &\Gamma_j\backslash\mathcal{H} & &[x_j]
\ar[u] &[x_j]
}$$
From this point of view the map induced by $g$ is trivial and the action
on bundles is trivial.  Hence for  $\varphi \in \mathbf{M}_k^U$
the pullback of  $f^{\varphi}_{g_j} \in M_k(\Gamma_j)$ by $g$ must be left
the same and we deduce that

$$
f^{\varphi}_{g_j}(hi)=(g^*\varphi)(g_jg^{-1}h)j(h,i)^{-k}deth^{-1} = \varphi(g_jh)j(h,i)^{-k}deth^{-1}
$$
for any $h \in Gl_2^+(\mathbb{R})$.  Clearly $g(\varphi)$ satisfies this condition and
we deduce that $g^*\varphi = g(\varphi)$. 
\end{proof}

\section{The $p$-adic Case}
\subsection{Overconvergent Modular Forms of Integer Weight} 
Fix once and for all an isomorphism $\mathbb{C} \cong \mathbb{C}_p$. Let
$\mathbf{A}^p_f$ be the finite adeles away from $p$.  For $U\in \Sigma$ let
$U^p$ denote its image in $\glafp$.  For $n \in \mathbb{N}$, define $\Gamma_1(p^n),\Gamma_0(p^n)
 \subset \glzp$ as the matrices which reduce to the mirabolic and borel respectively
 mod $p^n$.  Let us modify the set $\Sigma$
such that
\\ \\
$\Sigma:=\{ U \subset \glzh | \text{ open, compact subgroup  s.t. } U^p \times \glzp \cap \glpl\text{ is torsion free }\}$
\\ \\
As in the previous section, for any $n \in \mathbb{N}$ and $k \in \mathbb{Z}$,

$$ \varinjlim_{U \in \Sigma}\Gamma(X(\upi), \;\omega_{\upi}^{\otimes k})$$
comes with a left action of $\glafp$.
The complex smooth projective curve $X(\upi)$ can be considered as an
algebraic variety over $\mathbb{C}_p$ via the above isomorphism.  It can be rigidly analytified
to give the smooth projective rigid space $X(\upi)^{an}$ over $\Cp$.  The
sheaf $\omega_{\upi}$ can also be analytified and we denote this by $\omega_{\upi}^{an}$.
The functorial nature of analytification together with the rigid GAGA principle
tells us that there is an isomorphism of $\glafp$ modules

$$\varinjlim_{U \in \Sigma}\Gamma(X(\upi), \;\omega_{\upi}^{\otimes k}) \cong
\varinjlim_{U \in \Sigma}\Gamma(X(\upi)^{an}, \;(\omega_{\upi}^{an})^{\otimes k}).$$

$X(\upi)$ is  naturally the moduli space of elliptic curves with level $\upi$
structure and  consequently has a canonical model over $\mathbb{Q}$.  As such it can be
viewed as a nonsingular projective curve over $\Q_p$.  We can rigidly analytify
this curve to give a smooth projective rigid curve which we denote $X(\upi)^{an}/
\Qp$. Base change to $\Cp$ in the rigid analytic category recovers $X(\upi)^{an}$.
The points of $X(\upi)^{an}/
\Qp$ are elliptic curves (in the rigid or algebraic category) over a finite
extension of $\Qp$ together with appropriate level structure.

For an elliptic curve $E$ over a finite extension of $\Q_p$ there is a rational
number $V(E)$, defined using a lift of the Hasse invariant, which is a measure of the supersingularity of $E$.  If $0\leq V(E) < \frac{p^{(2-n)}}{p+1}$ then $E$ possesses
a canonical subgroup of order $p^n$.  For any $r\in\Q$ in this range we can define
$X(U^p \times \Gamma_1(p^n))_{\geq p^{-r}}/\Qp$ the admissible
open affinoid subspace of $X(U^p \times \Gamma_1(p^n))^{an}/\Qp$ whose non-cuspidal points correspond to elliptic curves
$E$ with $V(E)\leq r$, a level $U^p$ stucture and a point in the canonical
subgroup of order $p^n$. Similarly we define $X(U^p \times \Gamma_0(p^n))_{\geq p^{-r}}/\Qp$ as the admissible open affinoid subspace of $X(U^p \times \Gamma_0(p^n))^{an}/\Qp$
whose non-cuspidal points correspond to elliptic curves
$E$ with $V(E)\leq r$, a level $U^p$ structure and a subgroup of order $p^n$
which is the canonical subgroup. Geometrically, on any connected component of $X(U^p \times \Gamma_0(p))$ the supersingular locus is an annulus separating two ordinary
components, one containing $\infty$ and the other $0$.  The geometric connected
components of $X(U^p \times \Gamma_1(p))_{\geq p^{-r}}/\Qp$ are affinoid
subspaces of
these, which strictly contain
the ordinary locus with $\infty$ in some precise sense. We deduce that if $X(U^p \times \Gamma_1(p^n))^{an}/\Qp$
has $m$ connected components then so too does $X(U^p \times \Gamma_1(p^n))_{\geq p^{-r}}/\Qp$.  
These subspaces have been defined over $\Qp$ and we denote $X(\upi)_{\geq p^{-r}}/K$ as the base change to any complete extension $K$ of $\Qp$. We denote
the base change to $\Cp$
simply by $X(U^p \times \Gamma_1(p^n))_{\geq p^{-r}} \subset X(\upi)^{an}$.
If $U = \glzh$ we simply write $X_1(p^n)_{\geq p^{-r}}$ or $X_0(p^n)_{\geq p^{-r}}$.  These still make sense even if we are dealing with a coarse moduli
problem. If $p$ is odd let $\mathbf{q}=p$,
otherwise $\mathbf{q}=4$. 

\begin{defin} For $K/ \Qp$ a complete extension, $U \in \Sigma$, $k \in
 \mathbb{Z}$ and non-zero $r\in \mathbb{Q}$ in the appropriate range we define
$$\mathcal{M}_k(r,U^p, K) = \Gamma(X(U^p\times \Gamma_1(\mathbf{q}))_{\geq p^{-r}} /K,
(\omega_{\upq}^{an})^{\otimes k}),$$
the space of $r$-overconvergent
modular form over $K$ of tame level $U^p$ and weight $k$ with trivial action
of the Diamond operators at $p$.  
\end{defin}
Note that this demand on the Diamond operators means that we actually of
level
$\Gamma_0(\mathbf{q})$. We have phrased it in this way because once we have
redefined these spaces as functions using Eisenstein series in \S3.3 this will introduce a non-trivial action.

It is clear
that classical forms are overconvergent by restriction.  If we let $r$ vary
over the appropriate range then we get a direct system and we have 

$$ \mathcal{M}^{\dagger}_k(U^p, K) := \varinjlim_r \mathcal{M}_k(r,U^p, K),$$
the space of overconvergent modular forms of weight $k$ and tame level $U^p$.

For $V, U \in \Sigma$ open compact subgroups such that $V^p \subset U^p$
and $K/\Qp$ a complete extension  there is the functorial commutative diagram 
$$
\xymatrix{
X(V^p \times \Gamma_1(p^n))_{\geq p^{-r}}/K  \ar[r]^{i}  \ar[d] &X(U^p \times \Gamma_1(p^n))_{\geq p^{-r}}/K \ar[d]\\
X(V^p \times \Gamma_1(p^n))^{an}/K \ar[r] &X(U^p \times \Gamma_1(p^n))^{an}/K
}
$$
where the bottom horizontal arrow is just the usual inclusion functor. This
tells us that pullback of $r$-overconvergent forms of level $U^p$ and weight $k$ by the map $i$ gives $r$-overconvergent forms of level $V^p$ and weight $k$.
This gives the space of $r$-overconvergent modular
forms
over $K$
of integer weight a direct limit structure over $\Sigma$. We denote this space by $$\mathcal{M}_k(r, K):= \varinjlim_{U \in \Sigma}\mathcal{M}_k(r,U^p, K).$$
Similarly we have 
 $$\mathcal{M}^{\dagger}_k( K):= \varinjlim_{U \in \Sigma}\mathcal{M}^{\dagger}_k(U^p, K).$$
We drop the $K$ from the notation if $K=\Cp$. There is clearly an embedding of $\Cp$ vector spaces
$$ \mathcal{M}_k^{\Gamma_1(\mathbf{q})} = \varinjlim_{U \in \Sigma}\Gamma(X(\upq)^{an}, \;(\omega_{\upq}^{an})^{\otimes k}) \subset \mathcal{M}_k(r)  $$
\begin{prop}
There is an action of $\glafp$ on $\mathcal{M}_k(r)$ such that this embedding
is $\glaf$-equivariant.  
\end{prop}
\begin{proof}Let $g \in \glafp$ and $V, U \in \Sigma$ such that 
$V^p \subset gU^pg^{-1}$. We need to show that the functorial map

$$\xymatrix{X(V^p \times \Gamma_1(p^n))^{an} \ar[r]^g &X(U^p \times \Gamma_1(p^n))^{an}}$$ gives rise to the commutative diagram
$$
\xymatrix{
X(V^p \times \Gamma_1(p^n))_{\geq p^{-r}}  \ar[r]^{g}  \ar[d] &X(U^p \times \Gamma_1(p^n))_{\geq p^{-r}} \ar[d]\\
X(V^p \times \Gamma_1(p^n))^{an} \ar[r]^g &X(U^p \times \Gamma_1(p^n))^{an}
}
$$
If we
can show this over $\Qp$ then the result follows after base change. Pullback
by the above map will give the desired action. Let $K$
be a finite extension of $\Qp$.  If $E/K$ is an elliptic curve such that
$0\leq V(E) < \frac{p^{(2-n)}}{p+1}$ then
it is naturally a complex curve using the isomorphism fixed at the beginning
of \S 3.   If $T(E)$ is the Tate module of $E$ then let us fix an isomorphism
$$\mu: T(E) \cong \hat{\mathbb{Z}}^2$$ up to an action (on the right) of $V^p \times \Gamma_1(p^n) \subset \glzh$ which fixes a point 
in the canonical subgroup of order $p^n$.  If we tensor both sides with $\mathbb{Q}$
then $\mu$ extends to an isomorphism
$$\mu : T(E) \otimes \mathbb{Q} \cong \af^2,$$
up to an action of $V^p \times \Gamma_1(p^n)$.  This corresponds to a $K$-valued
point of \mbox{$X(V^p \times \Gamma_1(p^n))_{\geq p^{-r}}$.}  The action on the
moduli problem induced by $g$ is to send $\mu$ to 
$$\xymatrix{\mu: T(E) \otimes \mathbb{Q}\cong \af^2 \ar[r]^{\;\;\;\;\;\;\;\;\;\;g} &\af^2
},$$
where the second arrow is the isomorphism induced by $g$ acting on the right.
We will denote this map by $\mu_g$.  This will correspond to an elliptic
curve isogenous to $E$.  If $E'$ is a second elliptic curve which possesses
an isogeny onto $E$ which is prime to $p$ then $V(E)=V(E')$ and the image
of the canonical subgroup of order $p^n$ in $E'$ is the canonical subgroup
of order $p^n$ in $E$. Because we can decompose $g$ as the
product of an element in $\glpl$ and $U_1(p^n)$, the problem is reduced to
showing that for any $\gamma \in \glpl \cap \glzp$ and $\tau \in \mathcal{H}$, $E_{\tau} := \mathbb{C}/ \mathbb{Z}\tau \oplus \mathbb{Z}$ is isogenous to $E_{\gamma(\tau)}$ of order prime to $p$. Here
$$E_{\gamma(\tau)} \cong \mathbb{C} / (\mathbb{Z}(a\tau +b) \oplus \mathbb{Z}(c\tau
+ d)),
$$
where
$$\begin{pmatrix}a
&b \\ c &d \end{pmatrix} \in M_2(\mathbb{Z}) \cap \glzp
.$$
There is a natural isogeny 
$$\phi: E_{\gamma(\tau)}\rightarrow E_{\tau},$$
induced by the obvious inclusion of lattices. If $z = A(a\tau +b) +B (c\tau
+ d) \in ker(\phi)$ where $A, B \in \mathbb{R}$, then 
$$\begin{pmatrix}A
&B \end{pmatrix}\cdot \begin{pmatrix}a
&b \\ c &d \end{pmatrix} \in \mathbb{Z}_p^2 \cap \mathbb{Q}^2 \Rightarrow
(A,B) \in \mathbb{Z}_p^2 \cap \mathbb{Q}^2$$
Hence if  $z$ is non-trivial it cannot be in the $p$ torsion of  $E_{\gamma(\tau)}$. Hence $\phi$
is of order prime to $p$, and we are done.
\end{proof}
This action is defined on $\mathcal{M}_k(r, K)$ for any $K/ \Qp$ a complete
extension.  As for the space of classical modular forms, $$\mathcal{M}_k(r,K)^{U^p} = \Gamma(X(U^p\times \Gamma_1(\mathbf{q}))_{\geq p^{-r}/K} ,
(\omega_{\upq}^{an})^{\otimes k}).$$ 
All of these spaces contain the as subspaces cuspidal forms which we always
denote by $\mathcal{S}$ instead of $\mathcal{M}$.  The space of
cusp forms $\mathcal{S}_{k}^{\dagger} \in \mathcal{M}_{k}^{\dagger}$ is closed
under the action of $\glafp$.
Now we make the following important definition: 
\begin{defin}
If $f$ is a cuspidal overconvergent eigenform in the sense of \cite{CME}, then
define
 $$\pi_{f,l} := \Cp[\glql]f \subset \mathcal{M}_k(r)\subset\mathcal{S}_k^{\dagger}$$
\end{defin} 
Propositions 5 and the proof of proposition 6 immediately tell us that this agrees with the original
definition if $f$ is classical. The remainder of this section will be devoted
to generalising this to arbitrary weights.

\subsection{Rigid $q$-expansions}
Overconvergent modular forms of arbitrary weight are defined using the Eisenstein Family.  To generalise the above constructions we need to understand the
$q$-expansions of overconvergent modular forms of integer weight. As before,
let $K/ \Qp$ be a complete field extension.  Because

$$\mathcal{M}_k(r,K) = \varinjlim_{U(N)\in \Sigma} \mathcal{M}_k(r,K)^{U(N)^p}$$
it will be sufficient to determine $q$-expansions for forms of full level
$N$ structure where $(N,p)=1$.  As in the algebraic setting this can be done using the Tate
curve. 

Fix $K\subset \Cp $, a complete extension of $\Qp$, such that it contains the primitive $N^{th}$ root of unity $\zeta
= e^{\frac{2\pi i }{N}}$.
Also fix $f \in \mathcal{M}_k(r,K)^{U(N)^p}$ where $(N,p)=1$.
Let $S$ denote the punctured disc of radius 1 with parameter $q^{1/N}$ in the category of rigid analytic
spaces
over $K$.   As usual, $\mathbb{G}_m/\langle q \rangle$ is the Tate curve over
$S$.  A full level $N$ structure on this curve is an isomorphism of group objects in
the category of rigid spaces over $S$:
$$i_N:(\mathbb{Z}/N\mathbb{Z})^2 \cong \mathbb{G}_m/\langle q \rangle[N]$$
Because $K$ contains a primitive $N^{th}$ root of unity this can be specified
on points.  For example:
$$i_N:(a,b) \longrightarrow \zeta^a\cdot q^{b/N}$$
The canonical subgroup of the Tate
curve is $\mu_p$ which naturally embeds in $\mathbb{G}_m$.  Now consider the triple
$(\mathbb{G}_m/q^{\mathbb{Z}}, i_N, i_p)$, where $i_N$ is a full level $N$
structure as above and $i_p$ is an embedding of $\mu_p$ in $\mathbb{G}_m$.
The non-cuspidal points of \mbox{$X(U(N)^p \times \Gamma_1(\mathbf{q}))_{\geq p^{-r}}/K$}
have the usual functorial interpretation, so this triple induces the map:
$$
\xymatrix{
S \ar[r] &X(U(N)^p \times \Gamma_1(\mathbf{q}))_{\geq p^{-r}}/K
}
$$
which we denote $(i_N, i_p)$. There is a canonical differential on the Tate curve, $dt/t$, where $t$ is the parameter on $\mathbb{G}_m$.
The pullback of $f$ under this map gives a global differential on the Tate curve over $S$.  All such
global differentials are of the form $\mathcal{O}(S)(dt/t)^{\otimes k}$.
Because the parameter on S is $q^{1/N}$ we know that the global functions on it
are finite-tailed laurent series in this variable with coefficients in $K$. The holomorphicity condition at cusps ensures that the pullback of $f$ is a differential which extends to the whole disc.  The  $q$-expansion of $f$
at the cusp defined by $(i_N, i_p)$ is $f_{(i_N, i_p)}(q)\in K[[q^{1/N}]]$
given by 
$$(i_N, i_p)^*f = f_{(i_N, i_p)}(q)(dt/t)^{\otimes k}.$$

Recall that classically automorphic forms have $q$-expansions for each $g\in
\glaf$.  In this setting we can almost recover this. More precisely, observe
 that $U_1(\mathbf{q})\vartriangleright U(N)^p \times \Gamma_1(\mathbf{q})$ and $$U_1(\mathbf{q})/(U(N)^p \times \Gamma_1(\mathbf{q})) \cong GL_2(\mathbb{Z}/N\mathbb{Z}) $$
This group naturally acts on our level structures and for $g \in U_1(\mathbf{q})$
we denote the new level structure by $g(i_N, i_p)$.  There is a canonical
choice of level structure $(i_N^{can}, i_p^{can})$ such that for $f$ classical,
the $q$-expansion
given by this process agrees with the 1-$q$-expansion as defined in \S 2.
Similarly for  $g \in U_1(\mathbf{q})$ the $g$-$q$-expansion agrees with the $q$-expansion
given by $g(i_N, i_p)$.  

Because $det(U_1(\mathbf{q}))=\zh$ we are assured that a finite number
of choices of $g$ will give maps from $S$ into each component of $X(U(N)^p \times \Gamma_1(\mathbf{q}))_{\geq p^{-r}}/K$.
Any function on a connected rigid space which vanishes on an admissible affinoid
subspace is identically zero everywhere.  From this we deduce that $f$ is uniquely determined
by a finite number of these $q$-expansions.
 
\subsection{Eistenstein Series of Integer Weight}
We will adopt the notational conventions of \cite{BE} and \cite{BMF}.  Let $\mathcal{W}/\Qp$ denote weight space;  the moduli
space (in the category of rigid spaces) of continuous morphisms $\Zp^* \rightarrow \mathbb{G}_m$.  Hence $\mathcal{W}(\Cp)$
naturally bijects with continuous homomorphisms $\kappa: \Zp^* \rightarrow
\Cp^*$. Define $\mathcal{D} = (\mathbb{Z}/ \mathbf{q}\mathbb{Z})^*$. The components of $\mathcal{W}$ naturally biject with characters
of $\mathcal{D}$ and we denote the identity component by $\mathcal{B}$.   For $p\neq 2$ let $\tau: \mathbb{Z}_p^*
\rightarrow \Cp^*$ be the composition of reduction and the teichmuller character.
If $p=2$ define $\tau(d)=(-1)^{(d-1)/2}$. Define the character $\langle\langle
 . \rangle\rangle:\Zp^* \rightarrow \Cp^*$ by $\langle\langle d \rangle\rangle
 = d/ \tau(d)$.  Clearly $\langle\langle .\rangle\rangle \in \mathcal{B}(\Cp)$.
Fix a $(p-1)$-st root $\pi$ of $-p$ in $\Cp$.  For $s \in \Cp$ such that
$|s| \leq |\pi/q|$, $\langle\langle. \rangle\rangle^s$ makes sense as a character
in $\mathcal{B}(\Cp)$. Such characters correspond to a closed disc $\mathcal{B}^*
\subset \mathcal{B}$.  There is an isomorphism
of rigid spaces between $\Bs$ and
the closed disc in $\B$ centred at the origin of radius $|\pi|$,
given by the map 
$$ s \rightarrow (1+q)^s-1 $$  viewing $s$ as a parameter on $\Bs$.
The inverse map is given by 
$$t \rightarrow log(t+1)/log(1+q).$$For  $s$
in this range and $\chi$,
any $\Cp$ -valued character of finite order on $\Zp^*$, we define $(s,\chi)\in
\mathcal{W}(\Cp)$ by $(s,\chi)(d) =\langle\langle d \rangle\rangle ^s  \chi(d)$.
 Such characters are called arithmetic.
If $\chi = \tau^n$ for some integer $n$ we abbreviate this notation to $(s,\chi)
= (s,n)$.  

The Iwasawa $p$-adic $L$-function is non-vanishing on $\mathcal{B} \backslash
\{0\}$.  Hence for any $\kappa \in \mathcal{B}(\Cp)\backslash
\{0\}$ there is a corresponding formal $q$-expansion $E_{\kappa}(q)$ which
$p$-adically interpolates classical Eisenstein series.  More precisely, for
$k$ an even integer greater than 2 which is divisible by $(p-1)$, $E_{(k,k)}$ is the level $\Gamma_0(\mathbf{q})$,
$p$-deprived Eisenstein series, in the sense of \cite{BIF}.  If we set $E_{(0,0)}
= 1$ then we get the Eisenstein family $\mathbb{E}(q) \in O(\mathcal{B})[[q]]$. 

For $k \in \mathbb{Z}$, $E_{(k,0)}$ is an overconvergent
modular form of
weight $k$ and level $\Gamma_1(q)$ over $\Qp$. It has character $\tau^{-k}$
with respect to the Diamond operators at $p$. If $k$ is at least 1 then it
is classical.  Proposition 6.1 of \cite{BE} tells us that for some $r \neq
0$, $E_{(k,0)}$ is non-vanishing on $X(U^p \times \Gamma_1(\mathbf{q}))_{\geq p^{-r}}$
for all $U \in \Sigma$. This $r$ is independent of $k$. Fix such an $r$.  For 
$$ f \in \Gamma(X(U^p\times \Gamma_1(\mathbf{q}))_{\geq p^{-r}},
(\omega_{\upq}^{an})^{\otimes k}) \text{ where } U \in \Sigma,$$
$f/E_{(k,0)}$ is a rigid analytic function on $X(U^p\times \Gamma_1(p))_{\geq p^{-r}}$. Hence we get an embedding 
\begin{align*}
\theta_k:\mathcal{M}_k(r) &\longrightarrow  \varinjlim_{U \in \Sigma}O(X(U^p\times \Gamma_1(\mathbf{q}))_{\geq p^{-r}}) := \mathcal{M}(r)\\
f &\longrightarrow f/E_{(k,0)}.
\end{align*}

We do not refer to this space as the the space of weight zero overconvergent modular forms
because we make no demands about the action of the diamond operators.  The image
of the $\theta_k$ is precisely the space of functions where the diamond operators act
by $\tau^k$.

Let $g \in \glafp$ and $V, U \in \Sigma$ such that 
$V^p \subset gU^pg^{-1}$. For $f \in \mathcal{M}_k(r)^{U^p}$, 

\begin{align*}g(f)&=g^*(f)\\&=
g^*(E_{(k,0)} . f/ E_{(k,0)}) \\
&= g^*(E_{(k,0)}) g^*(f/ E_{(k,0)})\\
&=g(E_{(k,0)}) g(f/ E_{(k,0)}),
\end{align*}
where $g^*$ is the pullback under the functorial morphism induced by $g$
as in proposition 6.
We deduce that $\theta_k$ is not an embedding of $\glafp$ modules. To rectify
this we make the following definition:
\begin{defin}
For $k \in \mathbb{Z}$, $g \in \glafp$ and $f \in \mathcal{M}(r)$ define 
$$g_k(f) = g(E_{(k,0)})/E_{(k,0)} \times g(f).$$
\end{defin}  
\begin{prop}
This gives rise to an action of $\glafp$ on $\mathcal{M}(r)$ which we refer
to as
the weight $k$ action.
With respect to the weight $k$ action on $\mathcal{M}(r)$, $\theta_k$
is an embedding of $\glafp$ modules.
\end{prop} 
\begin{proof}
Immediate from the above.
\end{proof}
To extend this to arbitrary weights we generalise the twist factor
$g(E_{(k,0)})/E_{(k,0)}$.

\subsection{Overconvergent Modular forms of Arbitrary Weight}
In \S 3.3 we observed that for $k \in \mathbb{Z}$ we could define a \textit{weight} $k$ action on $\mathcal{M}(r)$ such that $\theta_k$ became an embedding
of $\glafp$-modules.  For $g \in \glafp$ the standard action was twisted
by the factor $g(E_{(k,0)})/E_{(k,0)}$.  We now generalise this to arbitrary
$\kappa \in \mathcal{B}(\Cp)$.
\begin{thm}
Let $g\in Gl_2(\Q_l)$, $l\neq p$ and $N\in \mathbb{Z}$ such that $U(N)\subset
g\glzh^p g^{-1}$.   There exists a unique 
rigid analytic function 
$$e_{g}\in O(\mathcal{B}\times X(U(N)^p \times \Gamma_0(p))_{\geq 1}/\Qp)$$ such that on any affinoid subdomain of $\mathcal{B}$ it is strictly overconvergent
for some $r \neq 0$, and such that for any $\kappa \in \mathcal{B}(\Cp)$
where $\kappa =(k,0)$ for $k \in \mathbb{Z}$,
$$e_{g,\kappa} =   g(E_{(k,0)})/E_{(k,0)}.$$

\end{thm}
\begin{proof}
Essentially this theorem is a generalisation of Coleman's construction of
Hecke operators in \cite{BMF}.
As such we follow Coleman's approach, first giving a proof at the centre of $\mathcal{B}^*$ and then extending to all of $\mathcal{B}$ using theorem 2.1 of \cite{CCS}. We will be making extensive use of $q$-expansions so must first work over
a finite extension of $\Qp$ and then use a descent argument. 

Fix $K\subset \Cp$, a finite extension of $\Qp$, such that it contains the primitive $N^{th}$ root of unity $\zeta
= e^{\frac{2\pi i }{N}}$.
From now on, unless otherwise stated, all rigid spaces will be over $K$.
Note that without loss of generality we may assume that $N=l^d$ for some
$d \in \mathbb{Z}$.   
Again following \S B.1 of \cite{BMF}, $\Cp$-valued points of $\Bs$ can naturally
by identified with arithmetic characters $(s,0)$.
As in \S3.3 we define $E_{(s,0)}(q)$ a formal $q$-expansion, which
$p$-adically interpolates classical Eisenstein series. 
For $t\in \mathbb{Q}$, such that $0 < t < |\pi /\mathbf{ q}|$ let $\mathcal{B}[0,t]$ be the closed
disc of radius $t$ in $\Bs$.  
We will first prove $g^*E_{(s,0)}/E_{(s,0)}$,
interpreted in the appropriate sense, is a function on the rigid space $\mathcal{B}[0,t]\times
\xu$ for some non-zero $r$ and $t$.  
Any such function comes equipped with a $q$-expansions at each cusp as explained
in \S3.2. We must show that there is a rigid function on $B[0,t]\times
\xu$ whose $q$-expansions are the same as those of our "formal" object $g^*E_{(s,0)}/E_{(s,0)}$.
Any function with this property is necessarily unique because it is totally
determined on a Zariski dense set (the classical points) on any of the components.
\\ \\ 
Define the weight 1, level $\Gamma_1(\mathbf{q})$
classical modular form over $\Zp$, $$E:=E_{(1,0)}$$ with character $\tau^{-1}$ for the action
of the diamond operators at $p$. Recall that using the conventions of \S2,
$E \in \mathcal{M}_1^{U_1(\mathbf{q})}(\tau)$.
At the cusp $\infty$ on $X_1(\mathbf{q})$ it has a $q$-expansion with the property
$$E(q)\equiv 1 \text{ mod} \mathbf{q}.$$
In particular for any $s \in \Bs(\Cp)$, $E(q)^s$ makes sense as a formal
$q$-expansion.   Theorem $B4.1$ of \cite{BMF} and the comment immediately
afterwards tell us that $$E_{(s,0)}(q)/E(q)^s$$ is the $q$-expansion at the
cusp $\infty$  of an invertible rigid analytic function  on $B[0,t]\times X_0(\mathbf{q})_{\geq
p^{-r}}$ for some $t$ in the above range and $r$ non-zero which we now fix. As explained in
proposition 6,  $g$ induces the morphism:
$$
\xymatrix
{\xu \ar[r]^{\;\;\;\;\;\;g} &X_0(\mathbf{q})_{\geq p^{-r}}
}
$$
This gives rise to the morphism: 
$$
\xymatrix
{B[0,t]\times\xu \ar[r]^{\;\;\;\;\;\;1 \times g} &B[0,t] \times X_0(\mathbf{q})_{\geq p^{-r}}
}.
$$
By pullback we may consider $g^*(E_{(s,0)}/E^s)$
as a rigid function on $B[0,t]\times\xu $.  
\\ \\
For $A \in (\mathbb{Z}/l^d \mathbb{Z})^*$ define $h^A$ to be the element of $U_0(\mathbf{q})$
trivial everywhere except $l$ where it is of form 
$$h^A_l =  \begin{pmatrix}A'
&0 \\ 0 &1 \end{pmatrix},  
$$where $A'$ is a fixed integer lift of $A$. 
Let $c_A$ be the cusp in $\xu$ determined by $h^A$ in the sense of \S3.2. For any component of $\xu$
there is an $A$ such that the cusp $c_A$ is on that component. Hence any
function on $\xu$ is uniquely determined by its $q$-expansions at these cusps.
\\ \\
For positive $k \in \mathbb{Z}$, $E_{(k,0)}$ and $E^k$ are classical of weight
$k$, level $U_1(\mathbf{q})$ and character $\tau^{k}$.  Let us adopt the
notational conventions of proposition 4.  Hence at $c_A$,
$$g^*E_{(k,0)}(q) = \tau^{k}(u) l^{-(n+m)+nk}E_{(k,0)}(\alpha q^{l^{n-m}})$$
and
$$g^*E^k(q) = \tau^{k}(u) l^{-(n+m)+nk}E^k(\alpha q^{l^{n-m}}),$$
Where $\alpha = e^{2\pi i (\frac{-a}{l^{m+r}})}$. We deduce that at $c_A$,

$$g^*(E_{(k,0)}/E^k)(q) = g^*E_{(k,0)}(q)/ g^*E^k(q) = \frac{E_{(k,0)}(\alpha
q^{l^{n-m}} ) }{E(\alpha q^{l^{n-m}})^k}.$$
Within $\mathcal{B}[0,t]$ there is a zariski dense set of positive integers,
hence we deduce that at $c_A$,
$$g^*(E_{(s,0)}/E^s)(q) = E_{(s,0)}/E^s(q^{l^{n-m}}\cdot \zeta) = \frac{E_{(s,0)}(\alpha
q^{l^{n-m}}\cdot )  }{E(\alpha q^{l^{n-m}})^s   }.$$

In a similar way $g^*(E)/E$ may be considered as a rigid function on $\xu$.
At $c_A$,

$$(g^*(E)/E) (q) = \tau(u) l^{-(n+m)+n} \frac{E(\alpha
q^{l^{n-m}})}{E(q)}.$$

Note that $\tau(u) = \tau(l^{-n})$ and hence the constant term of this $q$-expansion
is $l^{-(m+n)}\langle\langle
l^n\rangle\rangle$. We want to normalise this function so that on each component of $\xu$ there is a cusp with
$q$-expansion congruent to $1$ mod $\mathbf{q}$.  To do this we will make
use of the following proposition:

\begin{lem}
The $q$-expansions of $g^*(E)/E$ at the cusps $c_A$ all have the constant
term $l^{-(m+n)}\langle\langle
l^n\rangle\rangle$.
\end{lem}
\begin{proof}
It is enough to show that  if

$$ h^A_l g = \begin{pmatrix}l^m
&\frac{a}{l^r} \\ 0 &l^n \end{pmatrix}\cdot u , \text{ with } a, r, m, n \in
\mathbb{Z}, u \in U_0(\mathbf{q}),$$
then $n$ and $m$ and $\tau(u) = \tau(l^{-n})$ are independent of $A$.
This can be checked by elementary means.  
\end{proof}
Hence we may normalise $g^*(E)/E$ to give $f$, such that for every component
of $\xu$, $f$ has a cusp
with $q$-expansion congruent to 1 modulo $\mathbf{q}$.  
By the $q$-expansion principle such a function must reduce to $1$ on the components of the special fibre
of the Deligne-Rapaport/Katz-Mazur model of $X(U(N)^p \times \Gamma_1(\mathbf{q}))$
containing the reduction of the cusps in $X(U(N)^p \times \Gamma_1(\mathbf{q}))_{\geq
1}$.  Hence 
$$|f-1|_{ X(U(N)^p\times\Gamma_1(\mathbf{q}))_{\geq
1}}\leq |\mathbf{q}|.$$
By continuity 
$$|f-1|_{ X(U(N)^p\times\Gamma_1(\mathbf{q}))_{\geq
1}}= \lim_{r\rightarrow 0^+}|f-1|_{\xu}.$$
We conclude that for any $\epsilon\in \mathbb{R}, |\mathbf{q}|<\epsilon<
1$,
there exists a non-zero $r$ such that 
$$|f-1|_{ \xu}\leq \epsilon.$$
Fix such an $r$ and $\epsilon$.  Fix $t \in \mathbb{Q}^*$ such that  
$$\begin{pmatrix}
s \\ n
\end{pmatrix}T^n \rightarrow 0, \text{ as } n \rightarrow \infty, \text {
for
} |T| \leq \epsilon, |s|\leq t.$$
If we view $f$ as a rigid analytic function on $\xu$ then it makes sense
to talk about $f^s$ as a rigid analtyic function on $\mathcal{B}[0,t]\times
\xu$.  At $c_A$ it has $q$-expansion 
$$f^s(q)=\frac{E(\alpha
q^{l^{n-m}})^s}{E(q)^s}.$$
We may multiply this function by $E^s/E_{(s,0)}$ and $g^*(E_{(s,0)}/E^s)$
to get $F$.  At $c_A$ we have

$$F(q) = \frac{E_{(s,0)}(\alpha q^{l^{n-m}})  }{E(\alpha q^{l^{n-m}})^s   }\cdot \frac{E(\alpha q^{l^{n-m}})^s}{E(q)^s}
\cdot \frac{E(q)^s}{E_{(s,0)}(q)} = \frac{E_{(s,0)}(\alpha q^{l^{n-m}})}{E_{(s,0)}(q)}.$$
Multiply this function by $l^{-(m+n)}\langle\langle
l^n\rangle\rangle ^{s}$, to give a function $e_g$. At $k \in \mathbb{Z}$,
$k \in \mathcal{B}[0,t](\Cp)$ this function has the same $q$-expansion at $c_A$ as $g^*(E_{(k,0)})/E_{(k,0)}$ for all $A \in (\mathbb{Z}/l^d \mathbb{Z})^*$.

As observed on page 5 of \cite{CCS}
the Eisestein family over $\B$
has $q$-expansion in $\Lambda[[q]]$, where $\Lambda$ is the Iwasawa algebra.
Let $\Lambda_K$ denote the Iwasawa algebra over the ring of integers of $K$. Viewing $\kappa$
as a parameter on $\mathcal{B}$, $l^{-(m+n)}\kappa(l^n)$ becomes an element
of $\Lambda_K$.
We deduce that at any cusp $e_g(q) \in \Lambda_K[[q]]$. Now we can apply the
results of Thm 2.1 \cite{CCS}, which states that  $e_g$ must extend uniquely
over $\mathcal{B}$, replacing  $\Lambda$ with $\Lambda_K$ and $Z$ with
$X(U(N)^p\times \Gamma_1(\mathbf{q}))_{\geq 1}$.  The proof is identical to that given in \cite{CCS}. Note that by construction $e_g$ has trivial action of the Diamond
operators at $p$ so is of level $U_0(\mathbf{q})$.
\\ \\
Finally we must prove that $e_g$ descends to a function on $\mathcal{B}\times X(U(N)^p \times \Gamma_0(\mathbf{q}))_{\geq 1}/\Qp$. To do this we will use a Galois
descent argument.

We may assume that $K = \Qp(\zeta)$.  This is a Galois extension of $\Qp$
with Galois group $G= (\mathbb{Z}/l^d \mathbb{Z})^*$.  Let us fix an
affinoid subdomain $Y = Sp(A) \subset \mathcal{B}$ defined over $\Qp$.  Let $Y_K = Sp(A\otimes K)$ be the base change to $K$. 
When we restrict the function $e_g$ to $Y_K$ we know by compactness
that there exists a non-zero $r$ such that $e_g$ is a function on $Y_K
\times \xz /K$.  $\xz / \Qp$ is an affinoid space which we denote $Sp(B)$.
In this sense $e_g \in (A\hat{\otimes}B)\otimes K$, where both products are
over $\Qp$.  We wish to show that $e_g \in A\hat{\otimes}B$.  There is an
natural action of $G$ on $(A\hat{\otimes}B)\otimes K$ and any function which
is invariant under $G$ is in $A\hat{\otimes}B$. Hence we must show that for
any $\sigma \in G$, $e_g$ and $\sigma(e_g)$ are equal.  It is enough to show
that they have the same $q$-expansions.
\\ \\
Let $c$ be a cusp in $\xz$, corresponding to $h \in U_0(\mathbf{q})$. Let
$t \in K^*$, $|t| < 1$ and $\mu \in Y(K)$.  Such points naturally have an
action of $G$.  
The
data $$( \mathbb{G}_m/ \langle t^N\rangle, x, y, \mu),$$ where $x$ and $y$
are generators of the
$l^d$ torsion of $\mathbb{G}_m/ \langle t^N\rangle$,  gives a $K$-valued point of
$Y_K\times  \xz/ K$ which we denote by $\phi$. Note that we make no demands
on level structure at $p$ as it must be the canonical subgroup by construction.
We also demand that 
$$\begin{pmatrix}
x \\ y
\end{pmatrix} = \gamma \begin{pmatrix}
\zeta \\ t
\end{pmatrix}$$ 
where $\gamma$ is the reduction of $h$ modulo $U(N)^p \times
\Gamma_0(\mathbf{q})$.  Hence the $q$-expansion at $c$ tells us what $e_g$
evaluated at $\phi$ is.  
\\ \\
For any $\sigma \in G$ we have the following commutative
diagram:
$$
\xymatrix{
K\ar[r]^{\;\;\sigma} \ar[d]&K \ar[r]^{i}\ar[d]^{} &K \ar[r]^{\sigma^{-1}} \ar[d]& K
\ar[d]
\\
(A\hat{\otimes}B)\otimes K\ar[r]^{\;\;\sigma} &(A\hat{\otimes}B)\otimes K\ar[r]^{\;\;\phi}&K \ar[r]^{\sigma^{-1}} &K
}
$$
The composition of the bottom arrows in this diagram give a new $K$-valued
point of $Y_K\times  \xz/ K$ which we will call $\sigma^{-1}(\phi)$.  Moduli
theoretically this point corresponds to the data
$$( \mathbb{G}_m/ \langle \sigma^{-1}(t)^N\rangle, \sigma^{-1}(x), \sigma^{-1}(y), \sigma^{-1}(\mu)),$$
where$$\begin{pmatrix}
\sigma^{-1}(x) \\ \sigma^{-1}(y)
\end{pmatrix} = \gamma \begin{pmatrix}
\sigma^{-1}(\zeta) \\ \sigma^{-1}(t)
\end{pmatrix}.$$ 
We deduce that $\sigma(e_g)$ evaluated at $\phi$ is equal to $\sigma$ applied
to the evaluation of $e_g$ at $\sigma^{-1}(\phi)$.  If we can show that this
is equal to $e_g$ evaluated at $\phi$ we are done.

If $\sigma^{-1} = A \in (\mathbb{Z}/l^d \mathbb{Z})^*$ then the cusp corresponding
to $\sigma^{-1}(\phi)$ is is determined by $h^A \cdot h$. Let us call this
cusp $\sigma^{-1}(c)$.
If 
$$h_l\cdot g=\begin{pmatrix}l^m
&\frac{\alpha}{l^r} \\ 0 &l^n \end{pmatrix}\cdot u , \text{ where } a, r, m, n \in
\mathbb{Z}, u \in U_0(\mathbf{q})$$ 
then at $c$,
$$e_g(q) = l^{-(m+n)}\kappa(l^n)\frac{E_{\kappa}(\alpha q^{l^{n-m}} )}{E_{\kappa}(q)}.$$
Here $\alpha$ is some power of $\zeta$.  
Note that
$$h^A_l\cdot h_l\cdot g=\begin{pmatrix}l^m
&\frac{a\cdot A'}{l^r} \\ 0 &l^n \end{pmatrix}\cdot u' , \text{ where } \alpha, r, m, n \in
\mathbb{Z}, u' \in U_0(\mathbf{q})$$
Hence the $q$-expansion of $e_g$ at $\sigma^{-1}(c)$ is
$$e_g(q) = l^{-(m+n)}\kappa(l^n)\frac{E_{\kappa}(\alpha^{A'} q^{l^{n-m}} )}{E_{\kappa}(q)}= l^{-(m+n)}\kappa(l^n)\frac{E_{\kappa}(\sigma^{-1}(\alpha) q^{l^{n-m}} )}{E_{\kappa}(q)}.$$
The coefficients of the Eisenstein family are elements of $\Lambda$. The
same is true of the function $l^{-(m+n)}\kappa(l^n)$.  Any such function
evaluated
at $\mu$ is therefore the same as $\sigma$ applied to it evaluated at $\sigma^{-1}(\mu)$.
We deduce that $e_g$ and $\sigma(e_g)$ have the same $q$-expansion at $c$.
Therefore $e_g$ and $\sigma(e_g)$ are equal and we are done.
 
\end{proof}

For the remainder of this section let all rigid spaces be over $\Cp$.  An
arbitrary weight is a morphism of rigid spaces $\phi: Y \rightarrow
 \mathcal{W}$.  For ease of exposition let us assume that $Y =Sp(A)$ is affinoid
and the image of $\phi$
is contained in the component $\mathcal{W}_i$ defined by the
character $\tau^i$. For $U \in \Sigma$ we define the space of $r$-overconvergent
modular forms of tame level $U^p$ and weight $\phi$ to be the subspace 
$$\mathcal{M}_{\phi}(r,U^p) \subset  O(Y \times X(U^p\times\Gamma_1(\mathbf{q}))_{\geq
p^{-r}}),$$
whose elements have character $\tau^i$ with respect to the diamond operators at $p$.  Now let
$$\mathcal{M}_{\phi}^{\dagger}(U^p) =\varinjlim_{r} \mathcal{M}_{\phi}(r,U^p),$$
and define 
$$\mathcal{M}_{\phi}(r) := \varinjlim_{U\in\Sigma}\mathcal{M}_{\phi}(r,U^p),$$
$$ \mathcal{M}_{\phi}^{\dagger} := \varinjlim_{U\in\Sigma}\mathcal{M}_{\phi}^{\dagger}(U^p)$$ the space of $r$-overconvergent (resp. overconvergent) modular forms of weight $\phi$.
All these spaces are  $A$-modules and the come equipped with a (weight zero)
action of $\glafp$. Similarly we may replace all these spaces by the subspaces
of cusp forms $\mathcal{S}_{\phi}(r,U^p)$, $\mathcal{S}_{\phi}(r)$, $\mathcal{S}_{\phi}^{\dagger}(U^p)$
and $\mathcal{S}_{\phi}^{\dagger}$.  Both $\mathcal{S}_{\phi}(r)$ and $\mathcal{S}_{\phi}^{\dagger}$
are naturally (weight zero) $\glafp$-modules. Note that if $\phi$ is classical
(i.e. $\phi = k \in
\mathbb{Z}$)  then $\mathcal{M}_{\phi}^{\dagger}$ is precisely the image of
$\theta_k$ in \S3.3.
\\ \\
There is a natural isomorphism between $\mathcal{W}_i$ and the identity component $\mathcal{B}$. For $g \in \glql$ we can pullback $e_g$ along this isomorphism and then along $\phi$ to get 
$$e_{(g,\phi)} \in \varinjlim_{U\in \Sigma} \varinjlim_{r} O(Y \times X(U^p\times\Gamma_0(\mathbf{q}))_{\geq
p^{-r}}).$$

If $f\in \mathcal{M}_{\phi}^{\dagger}$ and $g \in \glql$ define   
$$g(f):=e_{(g, \phi)}\cdot g^*(f) \in \mathcal{M}_{\phi}^{\dagger},$$
where $g^*$ is the \textit{weight} $0$ action.  
\begin{prop}
This defines an action (weight $\phi$) of $\glql$ on $\mathcal{M}_{\phi}^{\dagger}$
(and on $\mathcal{M}_{\phi}(r)$ for some $r\neq 0$ dependent on $\phi$)
which is functorial with resect to the weight.
\end{prop}
\begin{proof}
Let $h, g \in \glql$.  Showing that this defines an action is equivalent to showing that
$e_{h}\cdot h^*(e_g) = e_{hg}$, where $h^*$ is the weight zero action. This is
true after restricting to any integer point of $\mathcal{B}$.  These are
Zariski dense in $\mathcal{B}$, hence they must have the same $q$-expansions
so are equal.

If $X=Sp(B)$ is an affinoid
rigid space  together with a morphism $\pi: X \rightarrow Y$. Then
pull back by morphisms of form 
$$\xymatrix{
X \times _{\Qp} X(U^p \times \Gamma_1(p))_{\geq p^{-r}}  \ar[r]^{\pi \times
1} &Y \times_{\Qp} X(U^p \times \Gamma_1(p))_{\geq p^{-r}} 
}$$
induces a map $\Pi:\mathcal{M}^{\dagger}_{\phi \cdot \pi} \rightarrow \mathcal{M}^{\dagger}_{\phi}$.
This maps is actually a $\glql$ module homomorphism.  This follows from the fact
that $(\pi \times 1)^* (e_{(g,\phi)})=e_{(g, \phi\cdot\pi)}$, where $(\pi
\times 1)$ is a
morphism at the appropriate tame level.

\end{proof}
This  action preserves the  subspace of overconvergent
cuspforms $ \mathcal{S}_{\phi}^{\dagger} \subset \mathcal{M}^{\dagger}_{\phi}$.
It is also true that for any $U\in \Sigma$, 
$$(\mathcal{M}^{\dagger}_{\phi})^{U^p} = \mathcal{M}_{\phi}^{\dagger}(U^p)$$
There is an analogous statement for $r$-overconvergent forms where the action
is defined.

Note that if $\phi$ is a classical weight and $f \in \mathcal{M}^{\dagger}_{\phi}$
is classical then the action of $\glql$ we have constructed agrees with the
classical one.
\section{$p$-adic analytic families of Admissible Representations of $Gl_2(\Q_l)$}
Fix $N$ a positive integer prime to $p$.   Recall that for $f \in \mathcal{S}_k^{U_1(N)}$, a cuspidal eigenform, we defined $\pi_{f,l} = \mathbb{C}[\glql]f \subset
\mathcal{S}_k$.  We are now in a position to generalise this to overconvergent
cuspidal eigenforms in the sense of \cite{CME}.  Keeping the conventions
of \S3.4, the $A$-module $\mathcal{M}^{\dagger}_{\phi}(U_1(N)^p)$ is equipped
with a set of commuting $A$-linear endomorphisms, commonly known as the Hecke
algebra $\mathbb{T}$. These preserve the space of cusp forms. The operator $U_p$ is contained in this algebra and
for $r$ where it is defined it acts compactly on the Banach space  $\mathcal{M}_{\phi}(r,U_1(N)^p)$.
In fact $U_p$ acts on the whole of $\mathcal{M}^{\dagger}_{\phi}$.
Fix $f \in \mathcal{S}^{\dagger}_{\phi}(U_1(N)^p)$, a normalised cuspidal eigenform
with respect to $\mathbb{T}$, such that the eigenvalue of $f$ with respect
to $U_p$ is a unit in $A$. This last condition  is commonly known as being finite
slope.   By the work
of \cite{CME}and \cite{BMF} this is equivalent to a commutative
diagram of rigid
spaces 

$$\xymatrix{
Y \ar[rr] \ar[dr]^{\phi} &&\mathcal{E} \ar[dl]
\\
&\mathcal{W}
}
$$
where $\mathcal{E}$ is the reduced tame level $N$ cuspidal eigencurve as constructed in \cite{BE}.
\begin{defin}
For $ f \in \mathcal{S}^{\dagger}_{\phi}(U_1(N)^p)$, a finite slope, cuspidal eigenform,
define

$$ \pi_{f,l} := A[\glql]f \subset \mathcal{S}^{\dagger}_{\phi},$$
where this is the weight $\phi$ action as defined in \S3.4.

\end{defin}
If $\phi$ and $f$ are classical then this recovers the previous definition.
Note that this is a smooth representation in the obvious sense.  

\begin{lem}Let $a_p \in A^*$, $U \in \Sigma$, then  $$\mathcal{M}^{\dagger}_{\phi}(U^p)^{U_p=a_p}
= \mathcal{M}_{\phi}(r, U^p)^{U_p=a_p},$$ for some non-zero $r$, which
depends only on $\phi$.
\end{lem}
\begin{proof}This the statement that all overconvergent modular forms of
fixed weight which
are finite slope eigenforms for $U_p$ are $r$-overconvergent for some fixed
$r$.   Following \cite{BMF},
there is an overconvergent function $e$ on the space $\mathcal{B}\times X(\glzh^p
\times \Gamma_0(\mathbf{q}))_{\geq
1}$ which is used to define $U_p$.  The $q$-expansion at infinity is $e(q)=
E_{\kappa}(q)/ E_{\kappa}(q^p)$. Hence, given  $U \in \Sigma$, 
$e$ is naturally an overconvergent
function on $Y \times X(U^p \times \Gamma_0(\mathbf{q}))_{\geq1}$ which we also denote $e$.  By compactness there exist a non-zero $r$ such that $e$ is a
function on $Y \times X(U^p \times \Gamma_0(\mathbf{q}))_{\geq p^{-r}}$.
 Let $U_0$ be the
\textit{weight} 0 $U_p$ operator defined by a correspondence on the moduli
problem. By propostion 3.5 of \cite{BA}, the operator $U_0$  increases overconvergence. If $f \in \mathcal{M}^{\dagger}_{\phi}(U^p)$ then by definition 
$$ U_p(f) = U_0(e \cdot f).$$ 
Hence, we deduce that if $f \in \mathcal{M}^{\dagger}_{\phi}(U^p)^{U_p=a_p}$,
then $$ f = a_p^{-1}U_0(e \cdot f),$$
which means that $f$ must be at least as overconvergent as $e$.  Hence we
are done.
 
\end{proof}
\begin{lem}
Let $f \in \mathcal{S}^{\dagger}_{\phi}(U_1(N)^p)$ be an eigenform whose
eigenvalue at $p$ is $a_p \in A^*$. Then 
$$\pi_{f,l} \subset (\mathcal{S}^{\dagger}_{\phi})^{U_p=a_p}.$$
\end{lem} 
\begin{proof}This statement will follow by showing that the action $\glql$ on $\mathcal{S}^{\dagger}_{\phi}$
commutes with the action of $U_p$. 

 Fix $g \in \glql$ and 
$V, U \in \Sigma$ such that $V^p \subset g U^p g^{-1}$ and both $e$ and $e_{(g,
\phi)}$ are functions on $Y \times X(V^p \times \Gamma_0(\mathbf{q}))_{\geq p^{-r}}$ for some non-zero $r$.  We will compare $U_p(g(f))$ and $g(U_p(f))$.
$$U_p(g(h)) = U_p(e_{(g,
\phi)}g^*(h))=U_0(e\cdot e_{(g,
\phi)}\cdot g^*(h)).$$
Conversely, 
$$g(U_p(h)) = e_{(g,\phi)} g^*(U_0(e\cdot h))= e_{(g,\phi)}U_0(g^*(e)\cdot
g^*(h)).$$
The second equality comes from the fact that  $U_0$ and $g^*$ are derived
from commuting actions on the moduli problem. A standard result of Coleman tells us that
$$e_{(g,\phi)}(q)U_0(g^*(e)\cdot
g^*(h))=U_0(e_{(g,\phi)}(q^p)g^*(e)\cdot g^*(h))$$
By considering $q$-expansions at any cusp we see that $e(q)\cdot e_{(g,\phi)}(q) = e_{(g,\phi)}(q^p)g^*(e)(q)$.
We deduce that the action of $\glql$ commutes with $U_p$.
\end{proof}
Fix $U \vartriangleleft \glzl$, an open compact subgroup such that  $U_1(N)^{l} \times U \subset U_1(N)$. Note that  $\pi_{f,l}^U
\neq 0$. 

\begin{prop} $\pi_{f,l}^U$ is a finitely generated $A$-module.
\end{prop}
\begin{proof} By lemma 10 and lemma 11, we know that $ \pi_{f,l} \subset
\mathcal{S}_{\phi}(r)^{U_p=a_p}$ for some non-zero $r$. In fact $$\pi_{f,l}
\subset \varinjlim_{V \subset \glzl} O(Y \times X(U_1(N)^{l,p} \times V \times
\Gamma_1(\mathbf{q})))_{\geq p^{-r}}.$$ where the limit is taken over appropriate
open compact subgroups. This is because $\glql$ only affects the moduli problem
at $l$. Hence for $V \in \Sigma$ such that
$V^p \subset U_1(N)^{l,p} \times U$, 
$$\pi_{f,l}^U \subset \mathcal{S}_{\phi}(r, V^p)^{U_p=a_p}.$$
Because $U_p$ acts compactly on this space, we know that the right hand side
must be a finite $A$-module.  Affinoid algebras are noetherian, hence $\pi_{f,l}^U$
is a finite $A$-module.

\end{proof}

We deduce that if $A=\Cp$ then $\pi_{f,l}$ is a smooth, admissible $\glql$-module
in the classical sense.
\\ \\
Fix $V \in \Sigma$ such that
$V^p \subset U_1(N)^{l,p} \times U$. Hence
$$\pi_{f,l}^U \subset \mathcal{S}_{\phi}(r, V^p)^{U_p=a_p}.$$   
The space $\mathcal{S}_{\phi}(r, V^p)$ is a potentially $ON$-able Banach space
(using the terminology of \cite{BE}) which compact
operator $U_p$.  If $P_{\phi}(T)$ is the characteristic power series of $U_p$
acting on $\mathcal{S}_{\phi}(r, V^p)$ then it naturally has $a_p^{-1}$ as
a root.  This data determines a closed embedding $\theta$ of $Sp(A)$ into the spectral
curve associated to $U_p$ acting on $\mathcal{M}_{\phi}(r, V^p)$. As in \cite{BE},
there is an admissible cover $\mathcal{U}$ of the spectral curve upon which we may construct
an eigenvariety by appropriately gluing together Hecke algebras.  Hence we
may admissibly cover $Sp(A)$ by affinoids such that each affinoid is contained
in some element of $\theta^{-1}(\mathcal{U})$.  Let us replace $A$ by an element
of this cover.  By definition of $\mathcal{U}$ we can find $Q(T) \in A[T]$ and $S(T) \in A[[T]]$ such that
$Q(T)$ and $S(T)$ are coprime,  $P_{\phi}(T)=Q(T)S(T)$ and $a_p^{-1}$ is
a root of $Q(T)$. We can also guarantee that the leading term of $Q(T)$ is
a unit.  By theorem 3.3 of \cite{BE},  $\mathcal{S}_{\phi}(r, V^p)$ decomposes
as the closed direct sum $\mathcal{S}_{\phi}(r, V^p)= N \oplus F$, where
$N$ is a finite, projective $A$-module.
$Q^*(U_p)$ acts invertibly on $F$ and as zero on $N$. We deduce that 
$$N = \mathcal{S}_{\phi}(r, V^p)^{Q^*(U_p)=0}.$$ But $a_p$ is a root
of $Q^*(T)$, hence $$\mathcal{S}_{\phi}(r, V^p)^{U_p=a_p} \subset N.$$ 
Let $\widetilde{\mathcal{E}}$ be a normalisation of $\mathcal{E}$
(over $\Cp$).  By the above we may admissibly cover $\widetilde{\mathcal{E}}$
by open affinoids $\phi:Sp(A)\subset \widetilde{\mathcal{E}}$ such that $A$
is a Dedekind domain and 
if $f$ is the eigenform coresponding to $\phi$ then 
$$\pi_{f,l}^U \subset \mathcal{S}_{\phi}(r, V^p)^{U_p=a_p}$$   
with
$$\mathcal{S}_{\phi}(r, V^p)= N \oplus F,$$where
$N$ is a finite, projective $A$-module upon which $U_p-a_p$ acts as zero.
Furthermore $U_p-a_p$ is never zero on $F$. Let us fix such a cover $\mathcal{V}$.
\\ \\
Now restrict to the case where $(Y, \phi) \in \mathcal{V}$. Note that $Y$ is an irreducible, smooth, connected, one dimensional
rigid space over $\Cp$.  
\begin{lem} $\pi_{f,l}^U$ is a finite, projective $A$-module.
\end{lem}
\begin{proof}$A$ being a Dedekind domain means that any finite $A$-module
is projective if and only if it
is torsion free.  Hence any submodule of a finite, projective $A$-module is
projective.  By construction
$$\pi_{f,l}^U \subset \mathcal{S}_{\phi}(r, V^p)^{U_p=a_p}.$$ 
By definition of $\mathcal{V}$ we know that the right hand side is contained
in $N$ which is finite and projective. Hence 

$$\pi_{f,l}^U \subset N,$$
and we are done.
Another way to see this is that the right hand side is a finitely generated $A$-module and each element in
it is determined by its $q$-expansions.  A function is zero if and only if
its $q$-expansions (which have coefficients in an integral domain)are all zero.  Hence, $\mathcal{S}_{\phi}(r, V^p)$ is
torsion free.

\end{proof}

Let $(Y, \phi) \in \mathcal{V}$.  Lemma 14  implies that any element of $End_A(\pi_{f,l}^U)$ has a trace in
$A$, which extends the usual definition for free modules. This trace map
is functorial in the following sense:  

Let $\kappa: A \rightarrow B$ be a
map of affinoid algebras over $\Cp$.  Then $\pi_{f,l}^U \otimes_{\kappa}B$
is a projective $B$-module.  $\Psi \in End_A(\pi_{f,l}^U)$ then $\Psi\otimes
1 \in End_B(\pi_{f,l}^U\otimes_{\kappa}B)$ and $tr(\Psi\otimes 1) = \kappa(tr(\Psi))$.

Now specialise to the case where $B = \Cp$.  In this situation $\kappa
\in Sp(A)$ so is given by a maximal ideal $m_{\kappa} \subset A$.  This will correspond to $\kappa
\in \mathcal{E}(\Cp)$.  By the functoriality (see proposition 10) of our construction there is
a $\glql$-equivarient morphisms of $\Cp$ vector spaces
$$ \xymatrix{ 
\mathcal{S}^{\dagger}_{\phi} \ar[r]^{\lambda} &\mathcal{S}^{\dagger}_{\phi \circ \kappa}
}.
$$   
This map preserves $r$-overconvergence.  We will be interested in the kernel
of this map.  
\begin{lem} $ker(\lambda)
= m_{\kappa} \cdot  \mathcal{S}^{\dagger}_{\phi}$.
\end{lem}
\begin{proof}Let $U\in \Sigma$ and $f \in \mathcal{S}_{\phi}(r, U^p)$.  $X(U^p
\times \Gamma_1(\mathbf{q}))_{\geq p^{-r}}$ is an affinoid over $\Cp$ whose associated affinoid
algebra we denote as $B$.  Hence $f \in A \hat{\otimes} B$.  
Let us choose a presentation of $A$ and $B$ in terms of Tate algebras.  By
this we mean isomorphisms
$$A \cong \Cp\langle T_1,...T_n\rangle / I, \; B \cong \Cp \langle S_1,...S_m\rangle
/ J,$$
for some $n, m \in \mathbb{N}$. Hence,
$$A \hat{\otimes} B \cong \Cp\langle T_1,...T_n, S_1,...S_m \rangle / (I,
J).$$
Let $f_{\kappa}$ denote
the image of $f$ in $\mathcal{S}_{\phi \circ \kappa}(r, U^p)$.  By defintion
$f_{\kappa} \in B$ and is formed by evaluating $T_1,...T_n$ at $\kappa$.
Now assume that $f \in ker (\lambda)$.
Let $$f' \in\Cp\langle T_1,...T_n, S_1,...S_m \rangle$$ be a lift of $f$,
then evaluating at $\kappa$ gives
$$f'_{\kappa} \in J \subset \Cp \langle S_1,...S_m\rangle,$$
which is a lift of $f_{\kappa}$.  Now $f' - f'_{\kappa}$ is also a lift of
$f$, whose evaluation at $\kappa$ is zero.  Hence $$f' - f'_{\kappa} \in m_{\kappa}\cdot
\Cp\langle T_1,...T_n, S_1,...S_m \rangle.$$  From this we deduce that $f
\in m_{\kappa}\cdot \mathcal{S}_{\phi \circ \kappa}(r, U^p)$ and that $$ker(\lambda)
\subset m_{\kappa}\cdot \mathcal{S}^{\dagger}_{\phi}.$$  The reverse inclusion is obvious,
hence we are done.
\end{proof}
Let $f_{\kappa}$ denote
the image of $f$ in $\mathcal{S}^{\dagger}_{\phi \circ \kappa}$.  Clearly
$f_{\kappa} \in \mathcal{S}^{\dagger}_{\phi\circ \kappa}(U_1(N)^p)$ is a
finite slope, cuspidal eigenform and so
we naturally have the $\glql$-module \mbox{$\pi_{f_{\kappa},l} \subset \mathcal{S}^{\dagger}_{\phi\circ\kappa}$.} Because $\kappa$
is surjective there is a surjection of $\glql$-modules
$$\xymatrix{ \pi_{f,l} \ar[r]^{\lambda} &\pi_{f_{\kappa},l}
}
$$
\begin{lem}The natural map
$$\xymatrix{ \pi_{f,l}^U \ar[r]^{\lambda} &\pi_{f_{\kappa},l}^U
}
$$
is a surjection.
\end{lem}
\begin{proof} 
There is a short exact sequence of $\Cp[U]$-modules:

$$\xymatrix{
0 \ar[r] &ker(\lambda) \ar[r] &\pi_{f,l} \ar[r]^{\lambda} &\pi_{f_{\kappa},l}
\ar[r]&0
}.$$
U is a profinite group so we can take continuous group cohomology to give
the long exact sequence:
$$\xymatrix{
0 \ar[r] & ker(\lambda)^U \ar[r] & \pi_{f,l}^U \ar[r]^{\lambda} & \pi_{f_{\kappa},l}^U
\ar[r]&  H_{con}^1(U, ker(\lambda)) \ar[r] &...
}$$
By definition $$H_{con}^1(U, ker(\lambda))= \varinjlim_{V\triangleleft U} H^1(U/V,ker(\lambda)^V),
$$
where the limit is taken over all open normal subgroups of $U$, which are
necessarily of finite index.  By Corollary
1 of \cite{AW} we know that each $H^1(U/V,ker(\lambda)^V)$ is annihilated
by $|U/V|$.  However $H^1(U/V,ker(\lambda)^V)$ naturally has the structure
of a $\Cp$-vector space and we deduce that $H^1(U/V,ker(\lambda)^V) = 0$
and therefore that $H_{con}^1(U, ker(\lambda)) =0$.
\end{proof}

\begin{prop}
There is a surjection of $\Cp$-vector
spaces:
$$ \xymatrix{\pi_{f,l}^U \otimes_{\kappa} \Cp \ar[r]^{\lambda} &\pi_{f_{\kappa},l}^U
}, $$
which for all but finitely many $\kappa \in Sp(A)$ is an isomorphism.
\end{prop}
\begin{proof}By lemma 16 this statement is equivalent to $ker(\lambda) = m_{\kappa}
\cdot \pi_{f,l}^U$ for all but finitely many $\kappa$.   By lemma 15 
$$ker(\lambda) = \pi_{f,l}^U \cap  (m_{\kappa} \cdot \mathcal{S}_{\phi}(r, V^p)).$$  
This can be further simplified to give 
$$ker(\lambda) = \pi_{f,l}^U \cap  (m_{\kappa} \cdot \mathcal{S}_{\phi}(r, V^p))^{U_p=a_p}.$$ 
By definition there is a decomposition $\mathcal{S}_{\phi}(r, V^p)= N\oplus
F$ with $U_p-a_p$ acting by zero on $N$. $U_p-a_p$ never acts as zero on
$F$.  Hence we
deduce 
$$(m_{\kappa} \cdot \mathcal{S}_{\phi}(r, V^p))^{U_p=a_p} \subset m_{\kappa}
\cdot N.$$ By construction $\pi_{f,l}^U \subset N$, so the statement of the
proposition is reduced to proving that 
$$m_{\kappa}
\cdot \pi_{f,l}^U =  \pi_{f,l}^U \cap (m_{\kappa}
\cdot N)$$ for all but finitely many $m_{\kappa} \in Sp(A)$.  This statement
is equivalent to proving that the natural map 
$$\xymatrix{ \pi_{f,l}^U \otimes_{\kappa} \Cp \ar[r] & N \otimes_{\kappa}
\Cp}$$
is infective for all but fintiely many $\kappa$.  Let $C$ be the cokernel
of the natural inclusion $\pi_{f,l}^U \subset N$.  Hence there is a short
exact sequence of $A$-modules:
$$\xymatrix{ 0 \ar[r] &\pi_{f,l}^U  \ar[r] &N \ar[r] &C \ar[r] &0
}$$
This gives rise to the long exact sequence:
$$\xymatrix{ ... \ar[r] &Tor_A^1(N, A / m_{\kappa}) \ar[r] &Tor_A^1(C, A / m_{\kappa}) \ar[r] &\pi_{f,l}^U\otimes_{\kappa}\Cp  \ar[r]&..}$$
$$
\xymatrix{ ..\ar[r] &N\otimes_{\kappa}
\Cp \ar[r] &C\otimes_{\kappa} \Cp \ar[r] &0
}.$$  
Hence the injectivity of this map is reduced to the vanishing of $Tor_A^1(C, A / m_{\kappa})$.  Computing Tor  using either the  left or right factor gives
the same result.
Hence the short exact sequence of $A$-modules:
$$\xymatrix{ 0 \ar[r] &m_{\kappa} \ar[r] &A \ar[r] &A/ m_{\kappa}\ar[r] &0
}$$
gives rise to the long exact sequence:
$$\xymatrix{ ... \ar[r] &Tor_A^1(C, A) \ar[r] &Tor_A^1(C, A/m_{\kappa}) \ar[r] & C \otimes_A m_{\kappa}  \ar[r]&..}$$
$$
\xymatrix{ ..\ar[r] &C \ar[r] &C\otimes_{\kappa} \Cp \ar[r] &0
}.$$  
$A$ is a flat $A$-module hence $Tor_A^1(C, A)$ is trivial.  We deduce that
$$Tor_A^1(C, A/m_{\kappa})=0 \Leftrightarrow C\otimes_A m_{\kappa} \rightarrow
C \text{ is injective. }$$
$C$ is a finitely generated module over a Dedekind domain.  The structure
theory of such modules is well understood.  If $C_{tor} \subset C$ denotes
the torsion submodule then 
$$C \cong C_{tor} \oplus C/ C_{tor}.$$
By construction the right hand summand is torsion free.  Hence it is projective
so flat.  We deduce that $Tor_A^1(C/C_{tor}, A/m_{\kappa})=0$.  $Tor$ respects
direct summands hence
$$Tor_A^1(C, A/m_{\kappa}) = Tor_A^1(C_{tor}, A/m_{\kappa}).$$
The general structure theory for finitely generated torsion modules over
a Dedekind domain implies that $C_{tor}$ is the finite direct sum of modules
of form $A/I$ for some ideal $I \subset A$. Thus we have further reduced
the problem to determining $Tor_A^1(A/I, A/m_{\kappa})$.  The same argument
as above implies that
$$Tor_A^1(A/I, A/m_{\kappa})=0 \Leftrightarrow A/I\otimes_A m_{\kappa} \rightarrow
A/I\text{ is injective. }$$
$A/I\otimes_A m_{\kappa}\cong m_{\kappa}/ I \cdot m_{\kappa}$, so this map
is injective if and only if $m_{\kappa} \cap I = I \cdot m_{\kappa}$.  This
is if and only if $m_{\kappa}$ and $I$ are coprime.  There are only finite
many maximal ideals diving $I$, hence $Tor_A^1(A/I, A/m_{\kappa})=0$ for
all but finitely many $m_{\kappa} \in Sp(A)$.  Because $Tor_A^1(C, A/m_{\kappa})$
is the direct sum of finitely many such groups we deduce the result.
Where $Tor_A^1(C, A/m_{\kappa}) \neq 0 $ there is a strict inclusion 
 $ m_{\kappa}
\cdot \pi_{f,l}^U \subset ker(\lambda)$ and all we know is that the $\lambda$
map in the statement of the proposition is a surjection.
\end{proof}
Fix a Haar measure on $\glql$.  Let $\mathcal{H}_U$ denote the Hecke algebra
of compactly supported, $U$ bi-invariant functions on $\glql$ with coefficients
in $\Cp$.   $\pi_{f,l}^U$ and $\pi_{f_{\kappa},l}^U$ are naturally $\mathcal{H}_U$-modules
and and the map $\lambda$ in lemma 16 is $\mathcal{H}_U$-equivariant . 
\begin{thm}
Let  $\widetilde{\mathcal{E}}$ be 
a normalisation of $\mathcal{E}$ (taken over $\Cp$).  There is a discrete subset $S \subset
\widetilde{\mathcal{E}}$
and a unique map $$\xymatrix{ tr_{aut}:\mathcal{H}_U
\ar[r]
&O(\widetilde{\mathcal{E}})},$$ such that for any $h \in \mathcal{H}_U$
and $\kappa \notin S$ the specialisation of $tr_{aut}(h)$ to $\kappa$
is equal to the trace of $h$ acting on  $\pi_{f_{\kappa},l}^U$.  In this sense
traces vary analytically, at least away from a discrete set.
\end{thm}
\begin{proof}
Fix for the moment an element of this cover $(Y, \phi) \in \mathcal{V}$.
Let $f_Y$ be the corresponding eigenform. By
lemma 14  $\pi_{f_Y, l}^U$ is a finite projective $A$-module. Fix $h \in \mathcal{H}_U$.  There
is a natural inclusion $h \in End_A(\pi_{f_Y, l}^U)$.  Hence such
an element possesses a trace in $A$.  The map $\lambda$ in lemma 16 is a
map of Hecke modules. By proposition 17 and the comments immediately preceding
lemma 15 the restriction of
this trace to all but finitely many $\kappa \in Sp(A)$ gives the trace of
$h$ acting on $\pi_{f_{\kappa}, l}^U$, where $f_{\kappa}$ is the restriction
of $f_Y$ to $\kappa$.  Hence $h$ gives rise to a function on $Y$ whose restriction
to all but finitely many points (independent of $h$) is the trace of $h$ acting on $\pi_{f_{\kappa}, l}^U$.

$\widetilde{\mathcal{E}}$ is a separated rigid space hence the
intersection of any two element of $\mathcal{V}$ is and affinoid curve. Any
two functions on an affinoid curve which agree at infinitely many points
must be the same.
Hence we deduce that the trace of $h$ lifts to a global function on $\widetilde{\mathcal{E}}$.
Letting $\mathcal{S}$ be the set of points for which the map in propostion
17 is not an isomorphism then we are done.
\end{proof}

\section{Galois Side}
\subsection{Galois Representations on the Eigencurve}
Let $\mathcal{E}$ denote the tame level $N$, cuspidal eigencurve over $\Qp$,
and $G_{\mathbb{Q}} = Gal(\overline{\mathbb{Q}}/ \mathbb{Q})$.
$\mathcal{E}$ naturally comes equipped with a universal continuous pseudocharacter:
$$\xymatrix{\mathcal{T}:G_{\mathbb{Q}} \ar[r] &O(\mathcal{E})
}.
$$
If $\kappa \in \mathcal{E}(\Cp)$ we denote the specialisation to $\kappa$
by $\mathcal{T}_{\kappa}$.   A theorem of R. Taylor (\cite{TAY})ensures
that this is the trace of a unique continuous semi-simple 2 dimensional  representation
$\rho_{\kappa}$, over $\Cp$.  Let $S$ denote the absolutely reducible
locus in the sense that $\rho_{\kappa}$  is absolutely reducible.  This is a discrete
subset and contains no classical cuspforms. Let $\mathcal{E}^o = \mathcal{E}\backslash
S$.  The classical locus of $\mathcal{E}$ is a Zariski dense subset. We deduce that the classical locus is still dense
in $\mathcal{E}^{o}$.  If $Y=Sp(A) \subset \mathcal{E}^o$ is an admissible
open subset then let $\mathcal{T}_Y$ be the restriction of $\mathcal{T}$
to $Y$. By lemma 7.2 of \cite{CHE} this is the trace of a unique, continuous
representation $$\xymatrix{ \rho_Y :G_{\mathbb{Q}} \ar[r] &B^*
}$$
where $B$ is a rank 4 Azumaya algebra over $A$.  By the definition of an
Azumaya algebra, $B$ is locally a matrix algebra.  We deduce that $Y$ has
an  admissible open affinoid cover, $\mathcal{V}$, such that for any $X \in \mathcal{V}$,
$\mathcal{T}_X$ is the trace of a unique (over $O(X)$) continuous representation
$$
\xymatrix{\rho_X: G_{\mathbb{Q}} \ar[r] &GL_2(O(X))
}.
$$
Hence we can find an admissible open affinoid cover of $\mathcal{E}^o$ with
this property.  Let us fix such a cover, $\mathcal{U}$. 

Let $F/\Qp$ be a finite extension and $\kappa \in \mathcal{E}^o(F)$.  Let $\wl$ denote the Weil group at $l$.    By the above we can attach to
$\kappa$ a unique, continuous, 2 dimensional representation of $\wl$ over
$F$.  Grothendieck's Monodromy theorem (4.2.2 \cite{T}) allows us to attach a 2 dimensional Frobenius semi-simple
Weil-Deligne representation (4.1.2 \cite{T}) to such a $\kappa$.  Applying the local
Langlands correspondence (see \S 5.3) we get an  irreducible smooth
representation of $\glql$.  We will generalise this to arbitrary $\kappa
\in \mathcal{E}^o(\Cp)$ and study how these objects vary across $\mathcal{E}^o$.

\subsection{Weil-Deligne Representations in Families}
Grothendiecks Monodromy theorem (4.2.2 \cite{T}) provides a dictionary between
$p$-adic representations of $\gl = Gal(\overline{\mathbb{Q}_l}/ \mathbb{Q}_l)$ and Weil-Deligne representations of $\wl$.

More precisely,  fix a finite field extension $F/\Qp$, an inverse (geometric)
Frobenius element $\Phi$ and a nonzero
additive homomorphism
$t_p:\il \rightarrow \mathbb{Q}_p$. ($t_p$ is unique up to constant scalar multiple.)  If $(\rho, N)$ is an $n$-dimensional Weil-Deligne representation
defined over $F$ then we define a continuous (with respect to the $p$-adic
topology on $F$) representation
$$\rho_l:\wl \rightarrow GL_n(F)$$
according to $\rho_l(\Phi^n\cdot u) = \rho(\Phi^n\cdot u)exp(t_p(u)\cdot
N)$, for all $n \in \mathbb{Z}$ and $u \in \il$.  Grothendieck's theorem
allows us to got backwards (4.2.1
\cite{T}) giving a bijection between isomorphism classes
of n-dimensional Weil-Deligne representations over $F$ and isomorphism classes
of continuous n-dimensional representations of $\wl$ over $F$.   
\\ \\
Let $A/\Qp$ be a reduced, affinoid algebra and $\rho_l$
a continuous $A$-linear representation of $\wl$ on a free $A$-module of rank
$n$. Fixing a basis naturally gives a morphism
$$\rho_l:\wl \rightarrow GL_n(A).$$
At any $F$-valued point of
$Sp(A)$ we can invoke the above dictionary to get an $n$-dimensional Weil-Deligne
representation.  If we take an arbitrary $\Cp$-valued point then this procedure
breaks down because $\Cp$ is not even an algebraic extension of $\Qp$.  Instead
we generalise Grothendieck's construction to work across all of $Sp(A)$
simultaneously.   
\begin{prop}Let $\rho_l$ be a continuous $A$-linear representation of $\gl$
on the free $A$-module $M$ of rank $n$. There exist a nilpotent endomorphism
$N$ of $M$ such that $\rho_l(u) = exp(t_p(u)N)$ for $u$ in an open subgroup
of $\il$.
\begin{proof}     The proof is almost identical to Grothendieck's original which
can be found in the appendix of \cite{ST}. Let $|.|$ denote the spectral norm
($A$ is reduced) on $A$.  Note that this gives rise to an ultrametric on
$A$. Define $A^o = \{ a \in A ; |a| \leq 1 \} $.  Consider
the open subgroup $H:= 1+p^2M_n(A^o) \subset GL_n(A)$. After fixing a basis
for $M$, the inverse image of $H$ under $\rho_l$ is an open subgroup of $\gl$.
Hence there is a finite field extension $K/\mathbb{Q}_l$ such that if $G_K:=Gal(\bar{\mathbb{Q}}_p/K)$
then  $\rho_l(G_K) \subset H$.  Let  $\kappa \in Sp(A)(F)$ for some $F/\Qp$
and $O_F$ its ring of integers.  Specialising  $\rho_l$ at $\kappa$ and then
restricting to $G_K$ yields a continuous homomorphism
$$\rho_{l, \kappa}: G_K \rightarrow 1+p^2 M_n(O_F)$$
If $K^t$ is the maximal tame extension of $K$ then the wild ramification group
$P_K:=Gal(\bar{K}/ K^t)$ is pro-l.  Because the image of $G_K$ is necessarily
pro-p we deduce that $\rho_l$ is trivial on $P_K$.  If we further restrict
to $I_K$ we know that it factors through $I_K/P_K =Gal(K^t/K^{nr})$, where
$K^{nr}$ is the maximal unramified extension of $K$. This group is (non-canonically)
isomorphic to $\prod_{p\neq l}\mathbb{Z}_p$. The same argument as above shows
that $\rho_{l, \kappa}$ is trivial on all components of this product away
from $p$. This is true of any such $\kappa$ point, the set of which naturally forms a
zariski dense set of $Sp(A)$. Hence we deduce that restricted of $\rho_{l}$
to $I_K$ is determined
by a continuous homomorphism
$$\rho_{l}: \mathbb{Z}_p  \rightarrow 1+p^2 M_n(A^o).$$
$\mathbb{Z}$ is dense in $\mathbb{Z}_p$, hence this morphism is determined
by $\alpha:= \rho_l(1)$. There exists $c \in \Qp$ such that $ct_p(I_F)=\mathbb{Z}_p$
and $\rho_l(u) = \alpha^{ct_p(u)}$. Note that raising $\alpha$ to something
in $\mathbb{Z}_p$ makes sense because $\alpha \in 1+p^2 M_n(A^o)$. The log
map is a group homomorphism from $1+p^2 M_n(A^o)$ to $ M_n(A^o)$ defined
by power series as follows:
$$log(1+B) = B - \frac{B^2}{2}+\frac{B^3}{3}-\frac{B^4}{4}+....$$  If $1+B
\in 1+p^2 M_n(A^o)$ then there is map $exp$, which is defined at
the level of power series and has the property that $exp(log(1+B))=1+B$.
Now define $N:= c\cdot log(\alpha)$, then $\rho_l(u) = exp(t_p(u)N)$.
We can now do the usual trick of conjugating by $\Phi$ to show that $\rho_l(\Phi)N\rho_l(\Phi)^{-1}
= p^{-1}N$ and hence that $N$ is nilpotent.
  
\end{proof} 
\end{prop}
Given any such $\rho_l$ we can find an $N \in End_A(M)$ as in the above proposition. Define the representation $$\rho:\wl \rightarrow GL_n(A)$$ by $\rho(\Phi^n\cdot u) = \rho_l(\Phi^n\cdot u)exp(-t_p(u)\cdot
N)$.  By construction $\rho$ is continuous with respect to the discrete topology on $M$.  An elementary check shows that $\rho(w)N\rho(w)^{-1} =
|w|N$ for all $w \in \wl$, where $|.|$ is inherited from the local reciprocity
isomorphism, sending $\Phi$ to a uniformiser (see \S5.3).
 Hence $(\rho, N)$ is
a Weil-Deligne representation on $M$, interpreted in the appropriate sense. Specialising  to any $F$-valued of $Sp(A)$
recovers classical construction. One advantage of this approach is that specialising to any $\Cp$-valued point
of $Sp(A)$ gives a Weil-Deligne representation over $\Cp$.   Note that even
though $N$ may be nonzero it may specialise to zero at some $\kappa \in Sp(A)(\Cp)$.
\\ \\
Ultimately we are interested in applying the local Langlands correspondence
to such Weil-Deligne representations across $\mathcal{E}^o$. Let $Y=Sp(A) \in \mathcal{U}$. Then $\rho_Y$ restricted to $\wl$ is a representation to which
we can apply proposition 18.  Hence we get $(\rho_{Y,l},N)$, a 2 dimensional
Weil-Deligne
representation over $A$. This gives a family of Weil-Deligne representations
over all $\mathcal{E}^o$.

\subsection{Local Langlands and Local-Global compatibility}
The local Langlands correspondence for $GL_n$ over  a $p$-adic field $F$ (proven in
\cite{HT}) gives a natural bijection between isomorphism classes of irreducible
smooth representations of  $GL_n(F)$  (\cite{JL}) and $n$-dimensional $\Phi$-semisimple
Weil-Deligne representations of $W_F$ (\cite{T}).  The base field is generally taken
to be algebraically closed of characteristic 0. \cite{D} provides a fairly comprehensive
survey
in the case $n=2$. The bijection, though natural, can be normalised in different
ways. The first choice to be made is how to normalise the local
class field reciprocity isomorphism $W_F^{ab} \simeq F^*$.  We will adopt
the conventions of \cite{D}, where geometric Frobenius elements map to uniformisers.
If $|.|$ is the absolute value on $F$ then the local reciprocity map
induces a character on $W_F$ which we also denote $|.|$. The local reciprocity
isomorphism induces a character $|.|$ on $GL_n(F)$ after composing with $det$.
The unitary normalisation,
$\pi_u$,
\ is determined uniquely by stipulating that $L$ and
$\epsilon$ factor agree. The Tate normalisation, $\pi_t$, is defined by twisting
the unitary correspondence by $|.|^{\frac{1}{2}}$.   More precisely, if $\sigma$ is an Weil-Deligne
representation of the the above type over an algebraically closed field of
characteristic 0 then $\pi_u(\sigma)\otimes |.|^{\frac{1}{2}} = \pi_t(\sigma)$.
The Tate normalisation has the property that if $\sigma$ is defined over any field of characteristic 0 then so is $\pi_t(\sigma)$ (3.2.7 \cite{D})
in a way which is compatible with base change.
Let $\chi$ be a character of $W_F$. Using the local reciprocity isomorphism
$\chi\circ det$ becomes a character of $GL_n(F)$ which we also denote as
$\chi$. The Tate correspondence
is compatible with twisting by $\chi$ in the following way:
$$\pi_t (\sigma \otimes \chi) = \pi_t(\sigma)\otimes \chi.$$

Let $f$ be a classical normalised cuspidal eigenform of weight $k\geq 1$ and
level $N$ and Nebentypus $\chi$ defined over $\bar{\mathbb{Q}}_p$. Attached
to such an $f$ is a continuous two dimensional representation $\rho_f$ of the absolute
Galois group $G_{\mathbb{Q}}$ over $\bar{\mathbb{Q}}_p$ characterised by the following conditions: it is unramified outside $Np$ and for $l$ not dividing
$Np$
the characteristic polynomial of a geometric Frobenius element is 
$$ X^2 + a_l X +\chi(l)l^{k-1};$$
where $a_l$ is the eigenvalue of the Hecke operator $T_l$ acting on $f$. If we restrict to a
decomposition group at $l \neq p$ then by results of the previous section
we can associate to $f$ a 2-dimensional $\Phi$-semisimple Weil-Deligne representation
of $\wl$ which we denote $\sigma_{f,l}$.

On the other hand, we can associate to $f$ an automorphic representation of $\glaf$ over $\bar{\mathbb{Q}}_p$,
albeit up to a choice of normalisation.  Our choice of normalisation is determined
by the fact that away from $pN$ the eigenvalues of classical and adelic Hecke operators agree. Such a representation factors into the restricted tensor
product of smooth irreducible representations of $\glql$ as $l$ ranges over
all finite places of $\Q$.  In this way we can associate to $f$ a smooth
irreducible representation of $\glql$ over $\bar{\Q}_p$, which we denote 
$\pi_{f,l}$. Work of Carayol (\cite{C}) and other shows that
$$\pi_t(\sigma_{f,l}) = \pi_{f,l} \text{ for } l \neq p$$
This is local-global compatibility away from $p$.  Given the appropriate
reformulation (2.1.2 \cite{E}) there is a similar result
for $l = p$ as proven by Saito (\cite{SAI}). 

In \cite{E} the classical local Langlands correspondence is modified at non-generic
smooth irreducible representations of $\glql$.   In the case of $GL_2$ this
is just the $1$-dimensional representations factoring through the determinant.
If $$\sigma \cong   (\begin{pmatrix}\chi &0 \\
0 &\chi |.|^{-1}
\end{pmatrix}, 0 ),$$
then define $\pi_m(\sigma) = \mathcal{B}(\chi|.|^{\frac{1}{2}}, \chi|.|^{-\frac{1}{2}}),$
where $ \mathcal{B}(\chi|.|^{\frac{1}{2}}, \chi|.|^{-\frac{1}{2}})$ is the
(unitary, see  \S 11.2 \cite{DI}) reducible principal series with
 $1$-dimensional quotient $\pi_t(\sigma)$. For all other $\sigma$, \mbox{$\pi_m(\sigma)
 = \pi_t(\sigma)$}.

\subsection{Local Langlands in Families}
We now study how the local Langlands correspondence behaves over a 1-dimension, irreducible, connected, smooth affinoid rigid space  $Y=Sp(A)$
over $\Cp$. This is equivalent to the affinoid algebra $A$ being a Dedekind domain. Our ultimate goal is to apply these result to a normalisation
of the eigencurve which can be addmissibly covered by such spaces. 

Let $K$ be the field of fractions of $A$. Fix $\sigma = (\rho, N)$, a Weil-Deligne representation on a free $A$-module of rank 2. If  $\kappa \in
Sp(A)(\Cp)$ then we can specialise $(\rho, N)$ in two ways: locally at $\kappa$
or generically at $K$.  If we $\Phi$-semisimplify we get two $2$ dimensional
Weil-Deligne representations:  $(\rho_{\kappa}, N_{\kappa}) := \sigma_{\kappa}$ and $(\rho_{K}, N_K):= \sigma_{K}$.  Invoking Local Langlands gives $\pi_m(\sigma_{\kappa})$
and $\pi_m(\sigma_{K})$, two smooth irreducible representations of $\glql$
over $\Cp$ and $K$ respectively. Let $\overline{K}$ be an algebraic closure
of $K$ and $\sigma_{\overline{K}} = (\rho_{\overline{K}},N)$ be the $\Phi$-semi-simplification of $\sigma
\otimes \overline{K}$.   By compatibility of $\pi_m$ with extension of scalars we know that
$\pi_m(\sigma_{\overline{K}}) \cong \pi_m(\sigma_K) \otimes \overline{K}$.
  $\pi_m(\sigma_{\overline{K}})$ can take
one of three forms: Principal Series, Supercuspidal or Special.  We deal with each in turn.

\subsection{Principal Series}
Fix a square root of $l$ in $\Cp$.   This case corresponds $\sigma_{\overline{K}} = (\rho_{\overline{K}},0)$, where $\rho_{\overline{K}}$ is reducible. More
precisely:
$$\rho_{\overline{K}}\sim \begin{pmatrix}\chi_1 &0 \\
0 &\chi_2

\end{pmatrix},$$
where $\chi_1$ and $\chi_2$ are $\overline{K}$-valued quasi-characters of $\wl$
such that $\chi_1/\chi_2^{-1} \neq |.|^{\pm 1}$. We ignore the degenerate
case because this cannot occur on $\mathcal{E}^o$.
By definition, $\chi_1$ and $\chi_2$ are continuous with respect to the discrete
topology on $\overline{K}$.  $\il$ is profinite so compact, hence $\chi_i(\il)$
is a finite group for $i=1, 2$.    We deduce that \mbox{$\chi_i(\il) \subset \overline{\Q}
\subset A$} for $i=1,2$. By construction we know that the characteristic
polynomial, $P_{\Phi}$, of $\rho_{\overline{K}}(\Phi)$ has coefficients in $A$, hence $$\chi_1(\Phi)\chi_2(\Phi)
\in A^*, \chi_1(\Phi)+\chi_2(\Phi) \in A.$$
There is no reason for either $\chi_1$ or $\chi_2$ to be defined over $A$
or even $K$.
We make use of the following proposition:
\begin{prop} Let $A$ be a reduced integral afffinoid algebra over $\Cp$,
$K$ its field of fractions.  Let $L$ be a finite field extension of $K$ and,
$B$ be  the integral closure of $A$ in $L$. Then $B$  is a finite $A$-module
and in particular is affinoid.
\begin{proof}Theorem 3.5.1 of \cite{FVP} tells us that the integral closure
of $A$ in $K$ is a finite $A$-module which we denote $R$.  $A$ is noetherian
and hence we deduce that $R$ is a noetherian, integrally closed domain. $B$
is naturally the integral closure of $R$ in $L$. $K$ is of characteristic
zero hence $L/K$
is separable and we can apply proposition
8 of \cite{SL} to deduce that $B$ is a finite $R$-module.  Hence $B$ is
a finite $A$-module.
\end{proof}
\end{prop}
Define $L$ to be the splitting field of $P_{\Phi}$ over $K$.  This is a Galois
extension of degree at most 2.  Let $B$ be the integral closure of $A$ in
$L$.  By the above proposition this is an affinoid algebra which is finite over $A$. By construction $\chi_1$ and $\chi_2$ are defined over $B$.  The
local Langlands correspondence in this case is
correspondence  $$\pi_m(\sigma_{\bar{K}}) = \pi(\chi_1|.|^{\frac{1}{2}}, \chi_2|.|^{\frac{1}{2}}),$$
where we are adopting the notation and conventions of \S 11.2 \cite{DI}.
On the right hand side we are naturally viewing  the $\chi_i$ as characters
of $\Q_l^*$ via the local reciprocity isomorphism.

Define $$\Pi_B = \pi_B(\chi_1|.|^{\frac{1}{2}}, \chi_2|.|^{\frac{1}{2}}):= \{f:\glql
\rightarrow B; f \in \pi(\chi_1|.|^{\frac{1}{2}}, \chi_2|.|^{\frac{1}{2}})\}.$$
This makes sense because $\chi_1(a)\chi_2(d) \in
B$, for all $a, d \in \Q_l^*$. 

Let $P$ be the Borel subgroup consisting of upper triangular matrices in
$\glql$,  and $P(\mathbb{Z}_l) =
\glzl \cap P$. Fix $U \vartriangleleft \glzl$, an open compact subgroup such that $\chi_1 \chi_2$ is trivial on $P(\mathbb{Z}_l)\cap U$. This implies
that $\Pi_B^U\neq 0$.
\begin{lem}
$\Pi_B^U$ is
a finite free $B$-module of whose rank $d$ is equal to the order of $P(\mathbb{Z}_l)\backslash
\glzl / U$. Moreover,
$$\Pi_B^U\otimes_B \overline{K}
\simeq \pi(\chi_1|.|^{\frac{1}{2}}, \chi_2|.|^{\frac{1}{2}})^U.$$
\end{lem}
\begin{proof}
 The Iwasawa decomposition
(4.5.2 \cite{BUM}) tells us that $\glql =P\cdot \glzl$. Hence $f\in \pi(\chi_1|.|^{\frac{1}{2}}, \chi_2|.|^{\frac{1}{2}})$ is uniquely determined by its behaviour on $\glzl$.
By the conditions imposed on $U$ we are free to choose the behaviour of $f$
on each element of the double coset space $P(\mathbb{Z}_l)\backslash
\glzl / U$.
\end{proof}
Let $\mathcal{H}_U$ be the Hecke algebra introduced in \S 4.
It naturally acts in a compatible way on both $\pi(\chi_1|.|^{\frac{1}{2}}, \chi_2|.|^{\frac{1}{2}})^U$ and $\Pi_B^U$.
Because $\pi(\chi_1|.|^{\frac{1}{2}}, \chi_2|.|^{\frac{1}{2}})$ is irreducible
we know that $\pi(\chi_1|.|^{\frac{1}{2}}, \chi_2|.|^{\frac{1}{2}})^U$ is
a simple finite dimensional $\mathcal{H}_U$-module (4.2.3 \cite{BUM}). Such a module
is determined up to isomorphism by its trace $tr:\mathcal{H}_U \rightarrow \overline{K}$.
By the above lemma we know that this map factors through $B$.
\begin{thm} If $tr:\mathcal{H}_U \rightarrow B$ is the trace of $\mathcal{H}_U$
acting on $\Pi_B^U$, then
it factors through $A$.  Moreover for any $\kappa \in Sp(A)(\Cp)$, the specialisation
of this map to $\kappa$ is the trace of $\mathcal{H}_U$ acting on $\pi_m(\sigma_{\kappa})^U$.
\end{thm}
\begin{proof}
Let $\tau \in Sp(B)(\Cp)$ be a point lying over $\kappa$, $m_{\tau}$ the
defining maximal ideal of $B$.  Let $\chi_{1,\tau}$
and
$\chi_{2,\tau}$ be the $\Cp$-valued quasi-character induced by $\tau$.  The
Weil representation $\chi_{1,\tau}\oplus
\chi_{2,\tau}$ has the same trace as $\rho_{\kappa}$, hence by the Brauer-Nesbitt
theorem we deduce that they are isomorphic. If $\mathcal{B}(\chi_{1,\tau}|.|^{\frac{1}{2}}, \chi_{2,\tau}|.|^{\frac{1}{2}})$ denotes the (potentially reducible) principal
series representation then there is a surjective morphism of $\glql$-modules:
$$\xymatrix{\Pi_B \ar[r] &\mathcal{B}(\chi_{1,\tau}|.|^{\frac{1}{2}}, \chi_{2,\tau}|.|^{\frac{1}{2}})
}.$$
The usual group cohomology argument (see lemma 16) tells us that the induced map:
$$\xymatrix{\Pi_B^U \ar[r] &\mathcal{B}(\chi_{1,\tau}|.|^{\frac{1}{2}}, \chi_{2,\tau}|.|^{\frac{1}{2}})^U
}$$
is surjective. $\Pi_B^U$ is finite and free by lemma 20.  We can choose a
basis whose image in $\mathcal{B}(\chi_{1,\tau}|.|^{\frac{1}{2}}, \chi_{2,\tau}|.|^{\frac{1}{2}})^U$
is a basis.  We deduce that the kernel is equal to $m_{\tau} \cdot \Pi_B^U$. Hence
there is an isomorphism
$$\Pi_B^U\otimes_{\tau} \Cp \cong \mathcal{B}(\chi_{1,\tau}|.|^{\frac{1}{2}}, \chi_{2,\tau}|.|^{\frac{1}{2}})^U.
$$
If $ \chi_{1,\tau}/\chi_{2,\tau} \neq |.|^{- 1}$, then $\pi_m(\sigma_{\kappa})
=\mathcal{B}(\chi_{1,\tau}|.|^{\frac{1}{2}}, \chi_{2,\tau}|.|^{\frac{1}{2}})$.
If $ \chi_{1,\tau}/\chi_{2,\tau} = |.|^{- 1}$ then $\pi_m(\sigma_{\kappa})
=\mathcal{B}(\chi_{2,\tau}|.|^{\frac{1}{2}}, \chi_{1,\tau}|.|^{\frac{1}{2}}).$
However, the semi-simplifications of $\mathcal{B}(\chi_{2,\tau}|.|^{\frac{1}{2}}, \chi_{1,\tau}|.|^{\frac{1}{2}})^U$ and $\mathcal{B}(\chi_{1,\tau}|.|^{\frac{1}{2}}, \chi_{2,\tau}|.|^{\frac{1}{2}})^U$  as $\mathcal{H}_U$-modules are isomorphic.
Hence they  the same trace.  We deduce that the specialisation to $\tau$ 
of the trace of $\mathcal{H}_U$ acting on $\Pi_B^U$ must equal the trace of $\mathcal{H}_U$
acting on $\pi_m(\sigma_{\kappa})^U$.  This is true of any $\tau$ lying over
$\kappa$, hence the trace is $Gal(L/K)$-equivariant so the image lies in
$K \cap B = A$.
\end{proof}

\subsection{Supercuspidal}
 If $\sigma_{\bar{K}} = (\rho_{\bar{K}}, 0)$, where
  $\rho_{\bar{K}}$ is irreducible then $\pi_m(\sigma_{\bar{K}})$ is supercuspidal.
By 2.2.1 \cite{T} we know that $\rho_{\bar{K}}$ is of the form $\nu \otimes
\chi$ where $\nu$ is an irreducible Weil representation of Galois type and $\chi$ is
an unramified character taking values in $\bar{K}$. Any irreducible Weil representation of Galois type on an algebraically closed field of characteristic zero necessarily
has finite image so is  defined over $\Cp$. Hence there is a Weil-Deligne
representation
$\sigma_{\Cp}$,
 defined over $\Cp$ such that  $\sigma_{\bar{K}} \cong (\sigma_{\Cp}\otimes
\bar{K})\otimes \chi$.
 Recall that the Tate normalisation
of local Langlands is compatible with base change and twisting.  Therefore,
$$\pi_m(\sigma_{\overline{K}}) = (\pi_m(\sigma_{\Cp})\otimes\overline{K})\otimes
\chi.$$
Although $\chi$ takes values in $\overline{K}$ we know that $\chi^2(\Q_l^*) \in
A^*$ by considering determinants.   Let $L$ be the splitting field of $x^2 - \chi(l)^2$
over $K$. Let $B$ be the integral closure of $A$ in $L$ and $R$ the integral
closure of $A$ in $K$.  $\chi$ is defined over $B$ and hence we can define the $B$-module $$\Pi_B:=(\pi_m(\sigma_{\Cp})\otimes
B)\otimes
\chi.$$
$\Pi_B$ is a free $B$-module because we are tensoring over a field.
\begin{lem}If $U \subset \glzl$ is an open compact subgroup then $\Pi_B^U$
is a finite free $B$-module. 
\end{lem}
\begin{proof}
$\chi$ is unramified so is trivial on $det(U)$.  Hence we need only consider $(\pi_m(\sigma_{\Cp})\otimes
B)^U$.  Clearly we have the inclusion
$$ \pi_m(\sigma_{\Cp})^U\otimes
B \subset (\pi_m(\sigma_{\Cp})\otimes
B)^U.$$
The left hand side is free because we are taking the tensor product over
a field.  Let $v \in  (\pi_m(\sigma_{\Cp})\otimes
B)^U$.  It is the sum of finitely many components of form $x \otimes b$,
where $x \in \pi_m(\sigma_{\Cp})$ and $b \in B$.  Because $(\pi_m(\sigma_{\Cp})$
is smooth we can choose $V \subset U$, an open subgroup of index $n$, which stabilises each
component in $v$.  If $g_1$, $g_2$....$g_n$ are right coset representatives
for $V$ in $U$ then by definition:

$$ v = \frac{1}{n}\sum_{i=1}^n g_i(v).$$
The right hand side is the sum of components of form $$\frac{1}{n}(\sum_{i=1}^n
g_i(x)) \otimes b. $$  By construction the left hand term in this product
is invariant under $U$ and we deduce that $$ \pi_m(\sigma_{\Cp})^U\otimes
B = (\pi_m(\sigma_{\Cp})\otimes
B)^U.$$
\end{proof}
\begin{thm} If $tr:\mathcal{H}_U \rightarrow B$ is the trace of $\mathcal{H}_U$
acting on $\Pi_B^U$, then
it factors through $A$.  Moreover for any $\kappa \in Sp(A)(\Cp)$, the specialisation
of this map to $\kappa$ is the trace of $\mathcal{H}_U$ acting on $\pi_m(\sigma_{\kappa})^U$.
\end{thm}
\begin{proof}
Let $\tau \in Sp(B)(\Cp)$ be a point  lying over $\kappa$ defined by the
maximal ideal $m_{\tau}$.  By the Brauer-Nesbitt
theorem $\rho_{\kappa} \cong \nu \otimes \chi_{\tau}$. Hence,
$$\pi_m(\sigma_{\kappa}) = \pi_m(\sigma_{\Cp})\otimes
\chi_{\tau}.$$
There is a surjective morphism of  $\glql$-modules 

$$\xymatrix{\Pi_B  \ar[r] &\pi_m(\sigma_{\Cp})\otimes
\chi_{\tau},}$$
induced by reducing by $m_{\tau}$.  We get
the short exact sequence
$$\xymatrix{0 \ar[r] & m_{\tau} \Pi_B\ar[r] &\Pi_B\ar[r] & \pi_m(\sigma_{\kappa}) \ar[r] &0
}$$
The by now familiar group cohomology argument gives the short exact sequence  of $\mathcal{H}_U$-modules
$$\xymatrix{0 \ar[r] & (m_{\tau} \Pi_B)^U\ar[r] &\Pi_B^U\ar[r] & \pi_m(\sigma_{\kappa})^U \ar[r] &0.
}$$
Because everything is smooth the same argument as used in lemma 23 tells
us that$(m_{\tau} \Pi_B)^U
= m_{\tau} \Pi_B^U$.   We deduce there is an isomorphism of $\mathcal{H}_U$-modules
:
$$\Pi_B^U \otimes_{\tau} \Cp \cong \pi_m(\sigma_{\kappa})^U.$$
Hence the specialisation to $\tau$ 
of the trace of $\mathcal{H}_U$ acting on $\Pi_B^U$ must equal the trace of $\mathcal{H}_U$
acting on $\pi_m(\sigma_{\kappa})^U$.  This is true of any $\tau$ lying over
$\kappa$, hence the trace is $Gal(L/K)$-equivariant so the image lies in
$K \cap B = A$.
\end{proof}

\subsection{Special}
The final case to consider is when $\sigma$ has non-trivial monodromy.  This
is the indecomposable  case (4.1.5 \cite{T}) and over $K$ 
$$\sigma_{K} \cong   (\begin{pmatrix}\chi|.|^{-1} &0 \\
0 &\chi 
\end{pmatrix}, \begin{pmatrix}0 &1 \\
0 & 0
\end{pmatrix} )$$ for some $K$-valued quasicharacter $\chi$. By construction,
the trace and determinant of $\rho_{\bar{K}}$ are contained in $A$. Hence
$\chi$ takes values in  $A^*$.  By  definition $\sigma$ is defined over $A$.
This means that $N \in End_A(M)$, where $M$ is free of rank 2.  It is possible
that $N$ will vanish after specialising to some $\kappa \in Sp(A)(\Cp)$.
Over $K$ however we can always find a basis such that $N$ has the above form.
What this means its that generic behaviour (over $K$) is too crude to tell
us about local behaviour everywhere.  However, if $A$ is 1-dimensional then
then $N$ can only specialise to zero at finitely many $\Cp$-valued points
of $Sp(A)$. Let us denote this set $S \subset Sp(A)$. 
\\ \\
In this case,
$$\pi_m(\sigma_{\bar{K}}) = \pi(\chi|.|^{-\frac{1}{2}}, \chi|.|^{\frac{1}{2}}),$$
Where the right hand side is the unique irreducible quotient of the reducible
principal series $ \mathcal{B}(\chi|.|^{-\frac{1}{2}}, \chi|.|^{\frac{1}{2}})$. Define the  $A$-modules $$\Delta_A:=\{f\in\mathcal{B}(\chi|.|^{-\frac{1}{2}}, \chi|.|^{\frac{1}{2}})\;\; |\;\; f(g) \in A, \forall g \in \glql\},$$
and 
$$\Omega_A := \{f \in \Pi_A \;\; | f(g) = \chi(det(g))a \text{ for some } a \in A\}$$
Note that $\Omega_A \neq 0$.
We are interested in the $A$-module $\Pi_A:= \Delta_A/ \Omega_A$. Note
that \mbox{$\Delta_A \otimes_A \bar{K} \cong \pi_m(\sigma_{\bar{K}})$. }
Fix $U \vartriangleleft \glzl$, an open compact subgroup such that $det(U)
\subset ker(\chi)$.  Hence $U$ acts trivially on $\Omega_A$.
\begin{lem}$\Pi_A^U$ is a finite projective  $A$-module.
\end{lem}
\begin{proof}
Let $f \in \Delta_A$ such that $f - u(f) \in \Omega_A$ for all $u \in U$.
Hence the image of $f$ in $\Pi_A$ is contained in $\Pi_A^U$.  This
is equivalent to the statement that for all $u \in U$, there exists $a_u
\in A$ such that $f(g)-f(gu) = \chi(det(g))a_u, \forall g \in \glql$.  This gives a function from
$\phi:U \rightarrow A$, $\phi(u) = a_u$. This is actually a group homomorphism
(with the additive group structure on A). Recall that by definition
$f$ is smooth, hence $\phi$ is trivial on some open subgroup of
$U$. $U$ is compact
so the image of $\phi$ is finite. But $A$ is a $\Cp$-vector space so $\phi$
is trivial.  Therefore $f \in \Delta_{A}^{U}$. By the same argument as lemma
21  $\Delta_{A}^{U}$ is free of rank
$|P(\mathbb{Z}_l)\backslash \glzl/ U|$.  Hence $\Pi_A^U$ is isomorphic 
to $\Delta_A^U / \Omega_A$, which is a finite, torsion free A-module, and
is hence projective. In fact it is  to be free of rank $|P(\mathbb{Z}_l)\backslash \glzl/ U|-1.$
\end{proof}
\begin{thm}
Let $\kappa \in Sp(A)$ and $U \in \glzl$ an open compact subgroup such that
$det(U)
\subset ker(\chi)$.  There is
an embedding of $\mathcal{H}_U$-modules 
$$\xymatrix{\Pi_A^U\otimes_{\kappa} \Cp \ar[r] &\pi_m(\sigma_{\kappa})^U}.$$
For $\kappa \notin S$, $\lambda$ is an isomorphism.  For $\kappa \in S$,  $\Pi_A^U\otimes_{\kappa} \Cp$ is the unique irreducible
$\mathcal{H}_U$ -submodule module of $\pi_m(\sigma_{\kappa})^U$corresponding to a Steinberg representation.  Hence if $tr:\mathcal{H}_U \rightarrow A$ is the trace of $\mathcal{H}_U$
acting on $\Pi_A^U$, then for any $\kappa \notin S$, the specialisation
of this map to $\kappa$ is the trace of $\mathcal{H}_U$ acting on $\pi_m(\sigma_{\kappa})^U$.
\end{thm}
\begin{proof} Let $m_{\kappa}$ be the maximal ideal defining $\kappa$. Specialising $\chi$
at $\kappa$ gives a surjective morphism of $\glql$-modules:

$$\xymatrix{\Pi_A \ar[r]  &\pi(\chi_{\kappa}|.|^{-\frac{1}{2}}, \chi_{\kappa}|.|^{\frac{1}{2}})}.$$
The usual group cohomology argument (see lemma 16) tells us that this induces a surjection
of $\mathcal{H}_U$-modules:
$$\xymatrix{\Pi_A^U \ar[r]  &\pi(\chi_{\kappa}|.|^{-\frac{1}{2}}, \chi_{\kappa}|.|^{\frac{1}{2}})^U}.$$
By the proof of lemma 25 we can find a basis of $\Pi_A^U$  whose image under
this map is a basis of $\pi_m(\sigma_{\kappa})^U$. We deduce that the kernel
is $m_{\kappa} \cdot \Pi_A^U$ and hence there is an isomorphism:
$$\Pi_A^U\otimes_{\kappa} \Cp \cong \pi(\chi_{\kappa}|.|^{-\frac{1}{2}}, \chi_{\kappa}|.|^{\frac{1}{2}})^U.$$

Assume $\kappa \notin S$.  In this case 
$$\sigma_{\kappa} \cong   (\begin{pmatrix}\chi_{\kappa}|.|^{-1} &0 \\
0 &\chi_{\kappa} 
\end{pmatrix}, \begin{pmatrix}0 &1 \\
0 & 0
\end{pmatrix} ).$$
Hence $\pi_m(\sigma_{\kappa})
= \pi(\chi_{\kappa}|.|^{-\frac{1}{2}}, \chi_{\kappa}|.|^{\frac{1}{2}})$ and
the above isomorphism becomes:
$$\Pi_A^U\otimes_{\kappa} \Cp \cong \pi_m(\sigma_{\kappa})^U.$$ The statement
about traces follows form the same arguments used in theorems 22 and 24.

Assume $\kappa \in S$.  In this case 
$$\sigma_{\kappa} \cong   (\begin{pmatrix}\chi_{\kappa}|.|^{-1} &0 \\
0 &\chi_{\kappa} 
\end{pmatrix},0 ),$$
and so $\pi_m(\sigma_{\kappa}) = \mathcal{B}(\chi_{\kappa}|.|^{\frac{1}{2}}, \chi_{\kappa}|.|^{-\frac{1}{2}})$.  This has $\pi(\chi_{\kappa}|.|^{-\frac{1}{2}}, \chi_{\kappa}|.|^{\frac{1}{2}})$ as its unique irreducible sub-representation
and hence there is an embedding 
$$\xymatrix{\Pi_A^U\otimes_{\kappa} \Cp \ar[r] &\pi_m(\sigma_{\kappa})^U}.$$
\end{proof}

\subsection{Local-Global Compatibility}
Using the constructions of the previous paragraph we will prove an analogous
result to Theorem 18 on the Galois side.

As indicated in \S5.1 there is an admissible affinoid cover of $\mathcal{E}^o$
(taken over $\Qp$)
such that each element of this cover comes equipped with a global Galois
representation.  Restricting to $\wl$ we may apply proposition 19 to each
element of this cover to get a family of 2 dimensional Weil-Deligne represention
across $\mathcal{E}^o$.  Base change to $\Cp$ and then consider a normalisation
which we denote $\widetilde{\mathcal{E}^o}$.  By the above construction it
has an  admissible affinoid cover which we also denote $\mathcal{U}$ (each member being 1-dimensional, connected,
irreducible and smooth), which has the property that for $Y=Sp(A) \in \mathcal{U}$
there is a Weil-Deligne representation over $A$, which agree on intersections.
$Y \in \mathcal{U}$ implies that $A$ is a Dedekind domain.  Hence the results
of \S 5.4 are applicable and we may apply the local Langlands correspondence
to each of these families of Weil-Deligne representations. To any $\kappa
\in \widetilde{\mathcal{E}^o}(\Cp)$ we therefore have a Weil-Deligne representation
$(\rho_{\kappa}, N_{\kappa})$ (over $\Cp$) and consequently an admissible, smooth representation
of $\glql$, $\pi_m(\rho_{\kappa}, N_{\kappa})$.

\begin{prop}
If $\widetilde{\mathcal{Z}} \subset \widetilde{\mathcal{E}^o}$ is a connected
component then for $\kappa \in \widetilde{\mathcal{Z}}(\Cp)$, $\pi_m(\rho_{\kappa}, N_{\kappa})$ is either generically (away from a discrete set) supercuspidal,
principal series or special.
\end{prop}
\begin{proof}
As in \cite{BC}, for $\kappa \in \widetilde{\mathcal{E}^o}(\Cp)$ the connected
component containing $\kappa$ is defined to the set of points $z \in \widetilde{\mathcal{E}^o}(\Cp)$
such that there is a finite set $Z_1,...,Z_n$ of connected admissible affinoid
opens in $\widetilde{\mathcal{E}^o}$ with $\kappa \in Z_1$, $z\in Z_n$ and
$Z_i \cap Z_{i+1}$ non empty for all $i$.  

Because any connected component $\widetilde{\mathcal{Z}} \subset\widetilde{\mathcal{E}^o}$ is an admissible
open there is a subset of $\mathcal{U}$ which admissibly covers $\widetilde{\mathcal{Z}}$.
On any element of this cover there is a family of admissible smooth representations
of $\glql$.  By construction they must agree on intersections.  $\widetilde{\mathcal{Z}}$
is separated so the intersection of two admissible connected affinoid opens
is again connected and affinoid. If two open affinoid subdomains have non-trivial
intersection in $\widetilde{\mathcal{Z}}$ then they must generically have
the same type of $\glql$ representation. Hence by the definition of a connected
component $\widetilde{\mathcal{Z}}$ must either be generically supercuspidal,
principal series or special.
\end{proof}
Let us follow to the conventions of \S4 fixing a Haar measure on $\glql$. Let $\mathcal{H}_U$ denote the Hecke algebra
of compactly supported, $U$ bi-invariant functions on $\glql$ with coefficients
in $\Cp$.    
Fix $U\vartriangleleft \glzl$, an open compact subgroup such that $U_1(N)^{l} \times U \subset U_1(N)$. This implies that $\pi_m(\rho_{\kappa}, N_{\kappa})^U$
is generically non-zero.

We will now prove the analogue of Theorem 18 on the Galois side.
\begin{thm}There is a map
$$\xymatrix{ tr_{Lan}:\mathcal{H}_U
\ar[r]
&O(\widetilde{\mathcal{E}^o})},$$
such that if  $\widetilde{\mathcal{Z}} \subset\widetilde{\mathcal{E}^o}$
is a connected component then either 

\begin{enumerate}
\item $\widetilde{\mathcal{Z}}$ is supercuspidal and for all $\kappa \in\widetilde{\mathcal{Z}}(\Cp)$, $tr_{Lan}$ restricted to $\kappa$ is the trace of $\mathcal{H}_U$  acting
on $\pi_m(\rho_{\kappa}, N_{\kappa})^U$.
\item $\widetilde{\mathcal{Z}}$ is special and for  $\kappa \in\widetilde{\mathcal{Z}}(\Cp)$
such that monodromy does not vanish $tr_{Lan}$, restricted to $\kappa$ is the trace of $\mathcal{H}_U$  acting
on $\pi_m(\rho_{\kappa}, N_{\kappa})^U$.  When monodromy does vanish,  $tr_{Lan}$ restricted to $\kappa$ is the trace of $\mathcal{H}_U$ acting on the $U$ fixed
points of the special subrepresentation of $\pi_m(\rho_{\kappa}, N_{\kappa})$.

\item $\widetilde{\mathcal{Z}}$ is Principal Series and and for all $\kappa \in\widetilde{\mathcal{Z}}(\Cp)$, $tr_{Lan}$ restricted to $\kappa$ is the trace of $\mathcal{H}_U$  acting
on $\pi_m(\rho_{\kappa}, N_{\kappa})^U$.
\end{enumerate}
\end{thm}
\begin{proof}
Parts (i) and (iii) are direct from Theorems 22 and 24 applied to each the
admissible cover of $\widetilde{\mathcal{Z}}$ described in the proof of proposition
27.  For part (ii) let $Y=Sp(A) \in \mathcal{U}$ be contained in $\widetilde{\mathcal{Z}}$.
Keeping the notation of \S5.7 we may construct $tr_{Lan}$ by gluing together
the trace functions of $\mathcal{H}_U$ acting on $\Pi_A^U$.  We know that
they agree on intersections because they must do generically.  For $\kappa
\in \widetilde{\mathcal{Z}}(\Cp)$ where monodromy doesn't vanish we are done
by theorem 26.  The case where monodromy vanishes is due to the construction
of $\Pi_A$ and our choice of normalisation of the local Langlands correspondence.
\end{proof}
There is a natural map from $\widetilde{\mathcal{E}^o}$ to $\widetilde{\mathcal{E}}$. Hence by pullback there is a map $$\xymatrix{ tr_{aut}:\mathcal{H}_U
\ar[r]
&O(\widetilde{\mathcal{E}^o})},$$ which satisfies the same properties as given in theorem 18.

\begin{thm} $tr_{aut} = tr_{Lan}$
\end{thm}
\begin{proof} It is enough to show that they agree on on connected components.
Coleman's classicality result tells us that there are a Zariski dense set
of classical points in any connected component $\widetilde{\mathcal{Z}} \subset\widetilde{\mathcal{E}^o}$.
Let $Y=Sp(A) \subset \widetilde{\mathcal{Z}}$ be a member of $\mathcal{U}$,
which contains a classical point.  Because classical points are never discrete
it must contain infinitely many such points.  By theorem 18 and theorem 28
there are a zariski dense set of classical points, $\kappa \in Y(\Cp)$, such
that $tr_{Lan}$ and $tr_{aut}$ restricted to $\kappa$  are the trace functions
of $\mathcal{H}_U$ acting on $\pi_m(\rho_{\kappa}, N_{\kappa})^U = \pi_t(\rho_{\kappa}, N_{\kappa})^U$ (classical forms always give generic representations) and
$\pi_{f_{\kappa},l}^U$ respectively. Here $f_{\kappa}$ is the cuspidal eigenform
corresponding to $\kappa$.  By classical local to global compatibility we
know that these two $\mathcal{H}_U$ modules are isomorphic, hence traces
must agree.  We deduce that on $Y$, $tr_{Lan}$ and $tr_{aut}$ agree on a
zariski dense set, and hence are equal.  By lemma 2.1.4 of \cite{BC} they must agree on
all of $\widetilde{\mathcal{Z}}$.  Hence we deduce that $tr_{Lan} = tr_{aut}$.
\end{proof}
Let $\mathcal{Z}$ be the irreducible component of $\mathcal{E}^o$ which is
the image of $\widetilde{\mathcal{Z}}$ under the natural projection of $\widetilde{\mathcal{E}^o}$
to $\mathcal{E}^o$. Let us also fix $U$ as above but together with the property
that it is \textit{bonne} in the sense of \cite{CB}. This property means
that that there is an equivalence of categories between finite vector spaces
over $\Cp$ with an action of $\mathcal{H}_U$ (such objects are determined
up to semi-simplification by traces) and admissible smooth representations
of $\glql$ over $\Cp$.  The equivalence is given by taking $U$ invariants.
Such subgroups form a basic for neighbourhoods of the identity so we may
choose such a subgroup which also satisfies the condition  $U_1(N)^{l} \times U \vartriangleleft U_1(N)$.  Fix such a $U$.  This condition implies that $\pi_{f_{\kappa},l}^U$
is non zero for all $\kappa \in \widetilde{\mathcal{E}^o}(\Cp)$.  We now come to our
central result.
\begin{thmAB*}
Away from a discrete set of points local to global compatibility holds on
$\mathcal{E}^o$, i.e.
$$\pi_m(\rho_{\kappa}, N_{\kappa}) \cong \pi_{f_{\kappa},l}$$ for all $\kappa
\in \mathcal{E}^o(\Cp)$ away from some discrete set.  More precisely if $\mathcal{Z}
\subset \mathcal{E}^o$ is an irreducible component then either:
\begin{enumerate}
\item $\mathcal{Z}$ is Supercuspidal (i.e.$\widetilde{\mathcal{Z}}$ is Supercuspidal)
In this case local to global compatibility holds on all $\mathcal{Z}$.
\item $\mathcal{Z}$ is Special (i.e.$\widetilde{\mathcal{Z}}$ is Special). In this case local
to global compatibility holds everywhere except at points where monodromy
vanishes.  For such $\kappa$, $\pi_{f_{\kappa},l}$ is the unique special irreducible sub-representation of $\pi_m(\rho_{\kappa}, N_{\kappa})$.
\item $\mathcal{Z}$ is Principal Series (i.e.$\widetilde{\mathcal{Z}}$ is Principal Series).  Here local to global compatibility holds except at points where
the ratio of the Satake parameters becomes $l^{ \pm 1}$.  At such points all
we know is that these is a smooth admissible representation $\pi$ and a $\glql$-module
surjection $\pi \rightarrow   \pi_{f_{\kappa},l}$ where the semisimplification of
$\pi$ if isomorphic to the semisimplification of $\pi_m(\rho_{\kappa}, N_{\kappa})$.
\end{enumerate}
\end{thmAB*}
\begin{proof}  Let us fix an irreducible component $\mathcal{Z}
\subset \mathcal{E}^o$.  In proving the result it will be enough to do it
on $\widetilde{\mathcal{E}^o}$, which surjects onto $\mathcal{E}^o$.  We will go through each case in turn.
\\ \\
In case (i) for any $\kappa \in \widetilde{\mathcal{Z}}(\Cp)$, we know by theorem
28 that $\pi_m(\rho_{\kappa}, N_{\kappa})^U$ is irreducible because $U$ is
\textit{bonne} and $\pi_m(\rho_{\kappa}, N_{\kappa})$ is irreducible.  By theorem 29 we know that the trace of $tr_{aut}$ restricted
to $\kappa$ is the trace of this irreducible $\mathcal{H}_U$ module.  We
know by proposition 17 that this irreducible module surjects ($\mathcal{H}_U$
equivariently) onto  $\pi_{f_{\kappa},l}^U$.  Hence they must be isomorphic
as $\pi_{f_{\kappa},l}^U \neq 0$.
Therefore $\pi_{f_{\kappa},l}^U$ and $\pi_m(\rho_{\kappa}, N_{\kappa})^U$
are isomorphic irreducible $\mathcal{H}_U$ modules.  Because $U$ is \textit{bonne} we deduce that
$\pi_{f_{\kappa},l} \cong \pi_m(\rho_{\kappa}, N_{\kappa})$.
\\ \\
In case (ii) when monodromy does not vanish an identical argument to above yields the result. At the $\kappa$ where monodromy vanishes a similar argument
works taking the $U$ fixed vectors of the special sub-representation of $\pi_m(\rho_{\kappa}, N_{\kappa})$.  
\\ \\
in case (iii) the same argument as above works where $\pi_m(\rho_{\kappa}, N_{\kappa})$ is irreducible (i.e. when the ratio of the satake parameters
is not $l^{ \pm 1}$.  When this does occur $\pi_m(\rho_{\kappa}, N_{\kappa})^U$
is reducible and by proposition 17 we know that there is a finite dimensional
$\mathcal{H}_U$ module which surjects onto $\pi_{f_{\kappa},l}^U$, whose
semi-simplification is isomorphic the the semi-simplification of $\pi_m(\rho_{\kappa}, N_{\kappa})^U$.  Because $U$ is \textit{bonne} this corresponds to an a smooth irreducible
$\glql$ representation $\pi$ which surjects onto $\pi_{f_{\kappa},l}$.  This
completes the proof.
\end{proof}

\bibliographystyle{abbrv}
\bibliography{thesis}
\end{document}